\documentclass[12pt]{amsart}

\usepackage{amsmath,amssymb,amsthm}
\usepackage{graphics}
\usepackage{enumerate}
\usepackage[dvips]{color}
\usepackage{version}
\usepackage{comment}
\usepackage{stmaryrd}
\usepackage[all]{xy}

\usepackage{todonotes}

\newtheorem{Thm}{Theorem}[section] 
\newtheorem{Lem}[Thm]{Lemma}

\newtheorem{Cor}[Thm]{Corollary}
\newtheorem{Cor.Conj}[Thm]{Corollary of Conjecture}

\newtheorem{Prop}[Thm]{Proposition}
\newtheorem{Prob}[Thm]{Problem}
\newtheorem{Conj}[Thm]{Conjecture}

\newtheorem{Ques}{Question}

\theoremstyle{remark}
\newtheorem{Rem}[Thm]{Remark}

\newtheorem{CauCon}[Thm]{Caution-Conjecture}
\newtheorem{Ex}[Thm]{Example}

\theoremstyle{definition}
\newtheorem{Def}[Thm]{Definition}

\newtheorem{Step}{Step}
\newtheorem{Stp}{Step}

\newtheorem*{ack}{Acknowledgements}

\newcommand{\R}{\ensuremath{\mathbb{R}}}
\newcommand{\C}{\ensuremath{\mathbb{C}}}

\newcommand{\Z}{\ensuremath{\mathbb{Z}}}

\newcommand{\Q}{\mathbb{Q}}

\newcommand{\A}{\mathbb{A}}

\newcommand{\X}{\mathcal{X}}

\newcommand{\OO}{\mathcal{O}}

\newcommand{\PP}{\mathbb{P}}

\begin{document}

\title
{Polystable log Calabi-Yau varieties and Gravitational instantons}

\author{Yuji Odaka}
\date{\today}

\maketitle

\begin{abstract}
Open Calabi-Yau manifolds and log Calabi-Yau varieties  
have been broadly studied over decades. Regarding them as 
``semistable'' objects, we propose 
to consider their good proper subclass, which we regard as certain 
{\it poly}stable ones, morally corresponding to 
{\it semistable with closed (minimal) orbits} as the 
classical analogue of GIT. 

We partially confirm that 
the new polystability 
seems equivalent to the existence of non-compact complete Ricci-flat 
K\"ahler metrics with small 
volume growths, notably many examples of 
gravitational instantons. 
Also, we prove some compactness or polystable reduction type 
results, partially motivated by bubbles of 
compact Ricci-flat metrics. 
\end{abstract}



\section{Introduction}

\subsection{History and motivation}

Since the celebrated existence theorem of 
compact Ricci-flat K\"ahler manifolds \cite{Yau}, 
sometimes under the name as the Calabi conjecture, 
there has been also 
plenty of nice works of constructing 
{\it non-compact complete} Ricci-flat K\"ahler metrics. 
Another origin is the group of various 
gravitational instantons in real four dimensions 
with rapid curvature decays, 
while their signature being Lorenzian \cite{Hawk, GH, EH}, 
in the context of general relativity. 
We are not able to make an exhaustive 
list of references right here due to its numbers, 
but our discussion to follow necessarily 
includes many important examples. The first age general existence theorems in K\"ahler setting 
seem to be 
due to the papers 
\cite{BK87, TY, TYII, BK} (cf., also \cite[p246]{Yau.survey}) outside {\it smooth} complement divisors in 
smooth K\"ahler manifolds. 

\vspace{2mm}

This paper means to be the first part of our attempt to 
lay an algebro-geometric foundation for 
these spaces, especially to allow the complement of {\it singular} 
divisors which are much subtler. 
Therefore, although we basically work over complex numbers $\C$, all algebro-geometric arguments work over 
an arbitrary algebraic closed field of characteristic $0$ at least. 
Such hope for presence of such algebro-geometric 
treatment could have been in the line along 
the conjecture in \cite[section 3, p246]{Yau.survey}, which 
predicted that 
any complete Ricci-flat K\"ahler metrics on $X^{o}$ may be 
compactifiable to a compact K\"ahler manifold $X$ whose 
complement is the support of an anticanonical divisor $D$. 
However, the conjecture of {\it loc.cit}  is not necessarily 
 true as 
discovered counterexamples by \cite{AKL, Goto1, Goto2, Hattori} etc, 
hence this paper restricts our attention to {\it compactifiable case} 
i.e., when $X^{o}$ can be written as complement of analytic closed 
subset of projective normal variety $X$. 

Combined with the Hironaka's log resolution in such compactifiable 
case, there should be 
a log pair $(X,D)$ such that $X$ is smooth, 
$D$ is simple normal crossing such that 
$X^{o}=X\setminus 
{\rm Supp}(D)$. We call such pair {\it log smooth} in 
this paper as we expect no confusion, taking our context to account. 
Also, just in this introduction, we suppose $D$ is a reduced ($\Z$-)divisor for simplicity of exposition, 
while we discuss with $\Q$-divisor $D$ from the next section. 

The 
notion of {\it dlt (divisorially log terminal)} pair of a variety $X$ and 
a boundary divisor $D$, 
introduced 
by Shokurov \cite[\S1]{Shokurov} and more clarified in 
\cite{Sza}, \cite[\S 2.3]{KM} 
(cf., also \cite[0-2-10]{KMM}, \cite{Fjnlt}), 
slightly extends the ubiquitous log smooth setup. 
We propose to 
use this notion of dlt pairs more systematically, 
in much wider context of studying limiting behaviour of 
canonical K\"ahler metrics, their singularities, or its non-compact versions among others. 
For the details of the notion, we refer to \cite[\S 2.3]{KM}, \cite[\S4]{Fjn}, 
\cite{Kol13} for instance. 
As a technical important tool, extension of the 
minimal model program for only klt singular setting (as the groundbreaking  
\cite{BCHM}) 
to even more singular setting will be effectively used as in this paper; 
allowing dlt, log-canonical 
or even semi-log-canonical 
singularities. See e.g.,  
\cite{FjnM, Amb, KK, Gona, Gonb, FG, HX, Fjn.lc, Has, HH} 
as 
foundational results and recent developments in such direction. 

We start with raising the following conjecture, 
partially to set the scene.

\begin{Conj}[Asymptotics]\label{asym.intro}
Suppose $X^{o}$ is a non-proper but separated (Hausdorff) 
variety which only has kawamata-log-terminal singularities 
(assume smoothness for simplicity, if not familiar) over $\C$ 
and 
admits a complete 
Ricci-flat weak K\"ahler metric $g$ 
in which we fix a base point $p$. We suppose that 
$${\rm vol}(B(p,r))\sim c r^{d},$$ 
for $r\to +\infty$ 
where $B(p,r)$ denotes the geodesic balls of radius $r>0$ with 
the center $p$, and call this $d$ as  
the volume growth dimension ${\rm vg}(X^{o},g)$. 

\begin{enumerate}

\item \label{veryopen}
if ${\rm vg}(X^{o},g)>\dim_{\C}(X)$ 
then the logarithmic Iitaka dimension 
(\cite{Iitaka}
\footnote{contains an interesting metaphor from the 70s by its author  
as follows 
in {\it loc.cit} \S11.16 \\ ``$V=\mathbb{P}^{2}
-\bar{D}$ 
is algebraic geometry for dimension $1.5$'' 
(translated from Japanese version)}) 
satisfies 
$\bar{\kappa}(X^{o})=-\infty$. 
In particular, any compactification $(X^{o}\subset)X$ 
is covered by rational curves i.e., 
uniruled
\footnote{For instance, if $X^{o}$ is a smooth surface, more strongly it holds that 
$\bar{\kappa}(X^{o})=-\infty$ is equivalent to that 
$X^{o}$ is dominated by images of $\A^{1}_{\C}=\C$. This is due to the works of 
\cite{MT1, MT2, KeMc}. }.

\item 
If $X^{o}$ is compactified to be a log pair $(X,D)$ 
of smooth projective variety $X$ and its simple normal crossing 
divisor $D(\in |-K_{X}|)$ such that $X\setminus {\rm Supp}(D)=X^{o}$, 
 then the following inequalities hold; 
$$
{\rm dim}_{\R} \tilde{\Delta}(D)={\rm dim}(X^{o})^{an}
\le {\rm vg}(X^{o},g)
\le \dim_{\R}(X)=2\dim_{\C}(X), 
$$ 
where $\tilde{\Delta}(D)$ means the dual intersection cone complex  
\footnote{Note that the ``common'' dual intersection complex is $\Delta(D)=
(\tilde{\Delta}(D)\setminus \vec{0})/\R_{>0}$} of $D$ 
and $(X^{o})^{\rm an}$ is the Berkovich analytification of $X^{o}$ 
for trivial valuation (cf., \cite{Ber90}). 

\end{enumerate}
\end{Conj}
Note that, 
the volume growth dimension ${\rm vg}(X^{o},g)$ 
is at most $2{\rm dim}_{\C}(X^{o})={\rm dim}_{\R}(X)$ 
by the Bishop-Gromov comparison and 
it is not necessarily integer. Indeed, in the classical examples of 
\cite[4.2]{TY}, 
$\frac{2\dim(X)}{\dim(X)+1}$ is attained. 
Also, although we excluded in the above situation, 
{\it compact} Ricci-flat K\"ahler varieties can be also regarded 
as those with volume growth dimension $0$. 

\begin{Rem}
For above \eqref{veryopen} case, 
the converse does {\it not} hold in general. For instance, 
appropriate products of known gravitational instantons 
easily give examples. Also, even an indecomposable (surface) example 
exists by the thesis of Hein \cite{Hajo}. Indeed, the complement of 
type II, III, IV degenerations (resp., type $I^{*}$ degenerations)
in rational elliptic surfaces 
with Tian-Yau metric \cite{TY} are shown to be such examples, 
in \cite[1.5 (ii)]{Hajo} (resp., \cite[1.5 (iii)]{Hajo}). 

More details are as follows. 
For any of those fibers of type II, III, IV, 
after passing to a log resolution, 
the log crepant boundary divisor supported on 
the fiber becomes ``combinatorially same'' i.e., 
they all become simple normal crossing with 
$4$ components with just one  
non-reduced component of multiplicity $2$ penetrating the others. 
On the other hand, we remark that 
the conical angle of the ALG type bubble in that case 
(\cite[Table 1]{Hajo}, \cite[Theorem 1.1]{CVZ19}) is the inverse of the ramification index 
for the necessary {\it semistable reduction}. 
Also, Type $I^{*}$ (minimal) degenerations are originally non-reduced. 
\end{Rem}

This paper mainly focuses on the case when they are 
non-compact and also the volume growth dimension 
should not exceed the original complex dimension of $X$. 
We wish to discuss the remained case in near future, 
and mean this paper to be the first for such series 
(We would like to call the property of 
complete Ricci-flat K\"ahler manifold being 
${\rm vg}(X^{o},g)>\dim_{\C}(X)$ {\it very open}-ness as a pun.)
Anyhow, there are various log Calabi-Yau manifolds $(X,D)$, 
where $X^{o}=X\setminus {\rm Supp}(D)$ are {\it not} known if it  admits any complete Ricc-flat 
K\"ahler metric. Recently, S. Sun asked a question if possible obstruction 
is detectable in a systematic 
algebro-geometric framework. 

Before defining our stability conditions to answer his question, 
we recall further background. The K-stability of polarized variety 
\cite{Tia97, Don02} has a version for log pairs introdued by 
\cite{Don11} which aims to understand 
existence of {\it edge-conical singular} (weak) K\"ahler-Einstein metrics 
with cone angle $2\pi \beta$, when the coefficient $1-\beta$ 
of the divisors are {\it strictly less} than $1$ (i.e., $\beta>0$). 
There is also the framework via 
intersection numbers \cite{OS}. One 
may think it would be natural to use this log K-stability notion, 
even in the case the boundary divisors have coefficients $1$, 
so we recall the following result, formally applying the definition 
of log stabilities \cite{Don11, OS} even for ``conical angle $0$'' 
which do not a priori make sense for metrics. 

\begin{Thm}[{cf., \cite[6.3]{OS}}]\label{OSrev}
Assume $X$ is a smooth projective variety, $(X, D)$ is a log smooth 
Calabi-Yau pair, i.e., $K_X+D$ is numerically 
equivalent to zero divisor with a polarization $L$. 
\footnote{We remark that ``polarization'' in our context refers to 
an ample line bundle unless otherwise stated. }
Then,  $((X,D),L)$ 
is logarithmically K-semistable 
if and only if coefficients of $D$ are at most $1$ 
but it can{\bf not} be log K-polystable unless $\lfloor 
D\rfloor =0$.

More generally, semi-log-canonical Calabi-Yau polarized pair 
$((X,D),L)$ is log K-semistable but it is log K-polystable 
if and only if the pair has only klt singularities. 
\end{Thm}

Therefore, the main issue is to refine the log stability notion.  
Another motivation for the author 
is for the moduli problem (see \cite{mod}, 
also cf., \S \ref{st.red.sec}). 
We leave the details of the stability conditions to 
section~\ref{stab.sec} while we give some hints here. 
We introduce following three variants with slightly 
different purposes, reflecting the 
subtlety of setup of discussing complete Ricci-flat 
K\"ahler metrics, and for later each usefulness. 

\begin{itemize}
\item (Definition~\ref{ops}) We first introduce {\it weak open K-polystability} 
of open Calabi-Yau pair, which concerns log test configurations 
with only irreducible degeneration fibers. 
As we show, one of our main points of the notion is 
it can be studied fairly systematically 
and easily  (unlike the usual K-stability!) 
in terms of log canonical valuations and log canonical centers. 
In particular, the set of test configurations for this are only 
{\it countably} infinite even over $\C$. 

Another point is that the definition 
also somewhat resembles the original Futaki's obstruction \cite{Fut}. 

\vspace{2mm}

\item (Definition~\ref{ops2}) A somewhat stronger notion 
{\it open K-polystability}, stronger than 
weak open K-polystability, means to detect the existence of 
complete Ricci-flat K\"ahler metric. 

\vspace{2mm}

\item (Definition~\ref{ops3}) The 
{\it strong open K-polystability}, stronger than 
open K-polystability, 
means to detect the existence of 
complete Ricci-flat K\"ahler metric which is limit of 
conical singular canonical metric with angle converging to $0$. 
The definition involves log K-polystability (\cite{Don11, OS}) here. 

\vspace{2mm}

\item We show some stable reduction type results (compactness theorem) in \S\ref{st.red.sec}, partially motivated by 
minimal non-collapsing pointed Gromov-Hausdorff limits of 
compact Ricci-flat K\"ahler metrics for special 
maximal degenerating holomorphic families. 

\end{itemize}

\subsection{Conjectures on metrics}

As mentioned above, one side of the main motivations for 
our attempt to introduce stability notions is following 
expectation on metric existence. At this stage, we might better call it speculation. 

\begin{Conj}[Non-compact Yau-Tian-Donaldson conjecture]\label{YTD}
Take an arbitrary a polarized dlt log Calabi-Yau pair $((X,D),L)$ 
with $\lfloor D\rfloor =\sum_i D_i$, $K_{X}+D\equiv 0$ and $L$ being ample on $X^{o}:=X\setminus {\rm Supp}(\lfloor D \rfloor)$. 
\begin{enumerate}
\item 
Then, if there is a complete Ricci-flat weak K\"ahler metric $g$ on 
$X^{o}$ of the volume growth dimension 
at most $\dim_{\R}(X)=2\dim_{\C}(X)$, 
whose K\"ahler class is 
$c_{1}(L)|_{X^{o}}$, 
$(X^{o},L^{o})$ is weakly open K-polystable. 

\item 
Furthermore, such a complete Ricci-flat K\"ahler metric with 
any base point is the pointed Gromov-Hausdorff limit of 
the conical singular weak cscK metric on some polarized 
dlt log Calabi-Yau 
compactification 
$((X,(1-\epsilon)D),L_{\epsilon})$ for $\epsilon\to 0$ with appropriate 
base points with $c_{1}(L_{\epsilon}|_{X^{o}})\to c_{1}(L|_{X^{o}})$ for 
$\epsilon \to 0$, if and only if $(X^{o},L^{o})$ is strongly 
open K-polystable. 
\end{enumerate}
\end{Conj}
This obviously reflects the following question. 
\begin{Ques}[Approximability / Metric compactifiability]\label{approx1}
For an arbitrary open complete Ricci-flat K\"ahler manifold $(X^{o},g)$ 
with the volume growth dimension at most ${\rm dim}_{\C}(X)$, 
is there a compactification $X^{o}\subset X$ 
such that $(X,X\setminus X^{o}=:D)$ is dlt, where $D$ is 
a reduced integral divisor, $g$ is the limit of a sequence of 
conical singular weak cscK metrics (cf., e.g., \cite{KZ, LWZ}) 
$g_{i}(i=1,2,\cdots)$ on $X$ 
whose singularities are supported in $D$, whose 
conical angles $\beta_{i}$ converge to $0$? 
If not true in general, what conditions on $g$ ensures such approximability? 
\end{Ques}
For a special case of the above question for Tian-Yau metric \cite{TY}, 
especially when $X$ is a Fano manifold and $D$ is smooth (cf., 
Theorem~\ref{ops.ex} \eqref{..1}, \eqref{.1app} later) 
the answer is expected to be affirmative 
in \cite[\S 6, before Conjecture 1]{Don11}. 
Also, a partial confirmation of Conjecture~\ref{YTD} below 
follows from a result of Berman \cite{Ber} stratightforwardly. 

\begin{Cor}[{of \cite[4.8]{Ber}}]
In the setting of Conjecture~\ref{YTD}, under further assumption 
that $L_{\epsilon}\equiv -(K_{X}+(1-\epsilon)D)$, 
existence of a sequence of conical singular weak K\"ahler-Einstein 
metrics for $((X,(1-\epsilon)D),L_{\epsilon})$ with 
$\epsilon \to 0$ (which is expected to automatically converge to a 
complete Ricci-flat weak K\"ahler metric on $X^{o}$) 
implies strongly open K-polystability of 
$(X^{o},L^{o})$. 
\end{Cor}

In general, there are many difficulties to work on K\"ahler geometry for 
open manifolds compared with compact case. 
For instance, even for relatively simple complete K\"ahler manifolds, 
$\partial \bar{\partial}$ lemma does {\it not} holds. 
(Nevertheless, some affirmative results are in 
\cite[Theorem1]{Del}, \cite[3.11, A.3]{CHI}, for instance.) 
We also remark there is another systematically studied class of 
non-compact complete canonical K\"ahler metrics; the 
K\"ahler-Einstein metrics with {\it negative Ricci curvature} 
i.e., hyperbolic metrics on 
smooth locus of Deligne-Mumford stable curves and its higher dimensional 
extension to 
KSBA varieties (semi-log-canonical models), 
whose existence (\cite{BermanGuenancia}) 
matches to corresponding algebro-geometric K-stability results (\cite{Oda, Od, OS}). 

\vspace{2mm}
This paper focuses on the Ricci-flat case and 
more precisely on concrete analysis of our new stability notions 
which provide many evidences to the above conjecture~\ref{YTD} on 
the existence of complete Ricci-flat K\"ahler metrics, 
and also explore an application to the 
non-collapsing limits of families of compact Ricci-flat K\"ahler metrics, 
and the moduli compactification problem: 
see our next subsection  
and 
\S\ref{st.red.sec}. 
The supporting evidences for the above conjecture \ref{YTD}
are obtained from some results of 
concrete analysis of the stability notion below. 

\begin{Thm}[{A weaker version of Theorem\ref{ops.ex}}]\label{ops.ex.intro}
Suppose $(X^{o},L^{o})$ is an open 
$n$-dimensional polarized Calabi-Yau pair, 
and its compactifying polarized smooth log Calabi-Yau pair is denoted as 
$((X,D),L)$. In this introduction, 
we assume $X$ is smooth and $D$ is a reduced (coefficients $1$)  
integral simple normal 
crossing divisor, for simplicity just in this introduction. (See more 
singular extensions 
in Theorem~\ref{ops.ex}.)

\begin{enumerate}

\item 
If $D$ is smooth, then 
$(X^{o},L^{o})$ is weakly open K-polystable for any $L$. 
If $X$ is a Fano manifold, then it is strongly open K-polystable. 
(We extend to slightly more singular case in 
Theorem\ref{ops.ex} \ref{.1}). 

\item If $X^{o}$ is algebraic torus, then strongly open K-polystable.

\item Moreover, if $X^{o}$ is semi-abelian variety, 
then strongly open K-polystable. 

\item If $G:={\rm Aut}^{o}(X)$ is reductive and 
$((X,D),L)$ is weakly open K-polystable, then 
$D$ is GIT polystable with respect to the 
$G$-action. 

\item cluster log surface (cf e.g. \cite{GHK}) is not 
even weakly open K-polystable with respect to 
any polarization. 

\item rational elliptic surface $X$ with $D$ a 
nodal reduced fiber of $I_{\nu} (\nu \ge 1)$ Kodaira type, 
then $(X^{o},L^{o})$ is strongly open K-polystable 
at least for some $L^o$. 

\item 
If $((X,D),L)$ is weakly open K-polystable and 
some irreducible component $D_{i}$ of $D$ 
satisfies that 
\begin{itemize}
\item 
$L|_{X\setminus D_{i}}\sim_{\Q} 0$ 
and 
\item $D_{i}$ is ample (e.g. when 
$\rho(X)=1$),
\end{itemize}
 then $((X,D),L)|_{X\setminus D_{i}}$ 
 is the affine 
cone of a certain $(n-1)$-dimensional 
dlt log Calabi-Yau pair $((X',D'),L'):=
((D_{i},\cup_{j\neq i}D_{j}\cap D_{i}),N_{D_{i}/X}).$

\end{enumerate}

\end{Thm}

The above results match with known existence of 
gravitational instantons such as: 
\cite{TY, BK}, \cite{Hajo} (cf., also \cite{CJL}), 
\cite{GChen, CCI, CCIII}, 
\cite{ACyl3,ACyln} for example.


\begin{ack}We would like to thank the following friends and colleagues 
for fruitful discussions and comments: 
K. Hattori, H-J.Hein, S. Honda, R. Kobayashi, H.Nakajima, 
Y.Oshima, C.Spotti, S.Sun 
and  J.Viaclovsky. 
The author is partially supported by 
KAKENHI 18K13389 (Grant-in-Aid for Early-Career Scientists), 
KAKENHI 16H06335 (Grant-in-Aid for Scientific Research (S)) 
during this research. 
\end{ack}


\section{The stability conditions}\label{stab.sec}

This section introduces 
various stability notions for ``polarized Calabi-Yau varieties'' 
in somewhat generalized senses which are allowed to be either  non-proper, 
or ``log pairs'' i.e., $(X,D)$ with $K_{X}+D\equiv 0$, or 
with non-normal reducible $X$. The author believes the 
number of variants for these stability notions below indirectly 
reflect the subtlety of 
handling general complete Ricci-flat K\"ahler metrics. 
Indeed, we use most of the notions 
introduced here in later sections. 

\subsection{Irreducible case}

Firstly, we make the following usual setup. 

\begin{Def}\label{CY.notation}
\begin{enumerate}
\item \label{notation1}
A polarized (dlt) 
log Calabi-Yau pair $((X,D),L)$ in this paper 
means 
\begin{itemize}
\item $K_{X}$ is $\Q$-Cartier (``$\Q$-Gorenstein'' condition), 
\footnote{We put it here for simplicity. 
It is not necessary for various part of the following discussions.}
\item $D$ is $\Q$-divisor such that the pair $(X,D)$ is dlt, $K_{X}+D\equiv 0$ (which is equivalent to  
$K_{X}+D \sim_{\Q} 0$ by \cite{FjnM, Gona}), 
\item the rounddown part of $D$ which we write as $\lfloor D \rfloor$, 
and its fractional part 
$\{D\}(=D-\lfloor D \rfloor)$ are both $\mathbb{Q}$-Cartier
\item $L$ is ample . 
\end{itemize}
We denote 
$X^{o}:=X\setminus {\rm Supp}\lfloor D\rfloor 
=X^{\rm klt}$ as its klt locus, a Zariski open subset, 
and call it {\it open locus} sometimes in this paper. 
We also set 
$L^{o}:=L|_{X^{o}}$. 
Any polarized log Calabi-Yau pair is 
log K-semistable as reviewed as 
Theorem~\ref{OSrev} (\cite[6.3]{OS}) above. 
If further $L=-K_{(X,\{D\})}:=-(K_{X}+\{D\})$, we call $((X,D),L)$ 
is {\it anti-log-canonically polarized}. 

Also, if we consider $((X,D),L)$ which satisfies all above condition 
except for the dlt property while $(X,D)$ is still log canonical (resp., 
semi-log-canonical), each time we write so as 
{\it log-canonical (resp., semi-log-canonical) 
polarized log Calabi-Yau pair. }

\item \label{oCY}
A non-proper polarized log  
pair $(Y,M)$ is said to be 
{\it open polarized Calabi-Yau pair} in this notes if it
is compactifiable to a polarized log Calabi-Yau pair 
$((X,D),L)$ which satisfies $\lfloor D \rfloor \neq 0$, 
associated with a fixed isomorphism 
$(X^{o},L^{o})\simeq (Y,M)$. 
\end{enumerate}
\end{Def}

\begin{Rem}\label{conn.compo}
Note that the above compactifications are not unique 
for a fixed open polarized Calabi-Yau pair as well-known. 
Nevertheless, as an easy obvious case of 
the Shokurov-Koll\'ar connectedness principle, between different such compactifications, 
still the set of the connected components of $\lfloor D\rfloor$ are canonically bijective to each other. 
They are at most two connected components of $\lfloor D\rfloor$ 
(cf., e.g., \cite[2.1]{FjnM}, \cite[5.3]{Gona}). 
\end{Rem}

Now we define weak polystability notion and prepare its 
foundation.

\begin{Def}[Weak open polystability]\label{ops}
\begin{enumerate}

\item \label{pltt} For a 
polarized log Calabi-Yau pair $((X,D),L)$, 
a log test configuration (see \cite[\S 5, \S 6]{Don11}, \cite[\S 3]{OS} 
for the definition) 
$((\mathcal{X},\mathcal{D}),\mathcal{L})$ 
is called {\it plt-type} in this paper, 
if $(\mathcal{X},\X_{0})$ is plt. 
\footnote{plt means purely log terminal. See 
\cite[\S 2.3]{KM} for details if not familiar. 
This condition in particular 
implies $\X_{0}$ is klt and there is associated 
dreamy valuation for $K(X\times \A^{1})$.}

\item 
Take an arbitrary log Calabi-Yau pair $(X,D)$ 
which is dlt so that for its any polarization $L$ log K-semistability holds by 
Theorem~\ref{OSrev} above. We denote the restricted polarization as 
$L^o:=L|_{X^o}$. 
A pair $(X^o,L^o)$ or a triple $((X,D),L^o)$ 
is said to be 
{\it weakly open K-polystable} 
if any plt-type 
test configuration $((\X,\mathcal{D}),\mathcal{L})$ of $((X,D),L)$ 
whose log Donaldson-Futaki invariant vanishes 
${\rm DF}((\X,\mathcal{D}),\mathcal{L})=0$ 
(\cite[\S 5, \S 6]{Don11}, \cite[\S 3]{OS}) 
is product log test configuration i.e., 
$((X,D),L)$-fiber bundle over $\mathbb{P}^{1}$. 
We will see in Proposition \ref{open.st2}
that the notion only depends on $(X^o,L^o)$ to 
$L$, hence the terminology. 

Furthermore, if the only product log test configuration is trivial one, 
i.e., when ${\rm Aut}((X,D),L)$ is finite, then we call the triple 
$((X,D),L^o)$ 
{\it weakly open K-stable}. 
\end{enumerate}
\end{Def}

The plt-type condition together with the automatic Cartierness 
property of $\X_{0}$ 
imply that $\X_{0}$ is klt and vice versa (\cite[\S 2.3]{KM}). 
If $L=-K_{X}$ modulo $\Q D$, 
then $\mathcal{L}\equiv -K_{\X/\mathbb{P}^{1}}$ modulo $\Q \mathcal{D}$ 
(cf., \cite[read 2.3, 2.4]{OSS}). 
The notion may remind the experts of the notion of ``special test configurations'' 
for anticanonically polarized $\mathbb{Q}$-Fano varieties 
(\cite{DT, Tia97, LX}) as examples, but note that we do not require 
$\mathcal{L}$ is anticanonical which will be important later on. 
Therefore, our notion strictly extends the notion of 
special test configuration of Fano varieties. 
Also, it is easy to see that if we would allow 
{\it log canonical} Calabi-Yau variety as $\X_{0}$, then 
\cite{TY} examples do not satisfy the corresponding polystability 
notion e.g. $\PP^{2}$ degenerating into the projective 
cone over elliptic curve.

\vspace{1mm}
We make following remark about the {\it product-ness} 
of log test configuration.

\begin{Prop}\label{aut.prod2}
The automorphism groups 
${\rm Aut}^{o}((X,D),L)$ and ${\rm Aut}^{o}(X^{o},L^o)$ 
are different for general triple $((X,D),L)$, while they coincide 
when $K_{X}+D=0$ (assuming the base field $k$ has uncountable order). 
${\rm Aut}^{o}(-)$ mean the identity connected component 
of the automorphism groups here and henceforth. 
\end{Prop}

\begin{proof}[proof of Proposition\ref{aut.prod2}]
As an example, of which the above two automorphism groups are different, 
simply we can take the affine space $X^{o}=\A^{n}$ so that 
$(X^{o},L|_{X^{o}})$ includes the whole polynomial 
ring $k[X_{2},\cdots,X_{n}](\ni f)$ 
as 
$$(a_{1},\cdots,a_{n})\mapsto 
(a_{1}+f(a_{2},\cdots,a_{n}),a_{2},\cdots,a_{n}),$$
hence far from being an algebraic group, 
as it would be of course finite dimensinal if so. 

\begin{Lem}\label{crep.class}
For any fixed dlt log Calabi-Yau pair $(X,D)$, 
there are only at most countably many log crepant 
$(X',D')$ with $X\setminus {\rm Supp}\lfloor D\rfloor=
X'\setminus {\rm Supp}\lfloor D'\rfloor$, 
modulo the isomorphisms. 
\end{Lem}

Note that for instance, there are countable infinite order of 
toric pairs for each fixed dimension more than $1$. 
\begin{proof}[proof of Lemma~\ref{crep.class}]
Since $(X',D')$ is log crepant to $(X,D)$ for the 
obvious reason, all the irreducible components of $D'$ 
are log canonical places for the log pair $(X,D)$. From 
Zariski lemma (cf., e.g., \cite[2.45]{KM}) and the dlt property of $(X,D)$, 
there are only countably many such log canonical places 
hence we conclude the proof. 
\end{proof}
We now return to continue the proof of Proposition 
\ref{aut.prod2}. We take 
a connected algebraic subgroup of ${\rm Aut}^{o}(X^{o},L|_{X^{o}})$ 
which we denote as $G^{o}$. Then, thanks to the cute lemma 
above \ref{crep.class}, $((X,D),L)$ also admit natural 
$G^{o}$-action so that $G^{o}\subset 
{\rm Aut}^{o}((X,D),L)$. Then, ranging $G^{o}$ 
finishs the proof. 
\end{proof}

\begin{Cor}[Product-ness]\label{prod}
Any plt-type log test configuration $((\mathcal{X},\mathcal{D}),\mathcal{L})$ 
of $((X,D),L)$ which contains a 
product test configuration of 
$(X^{o},L|_{X^{o}})$ as the complement of the closure of 
$D\times (\PP^{1}\setminus \{0\})$ 
is itself a product log test configuration of $((X,D),L)$.  
Further, it is determined by the open product test configuration 
associated in that way. 
\end{Cor}

\begin{proof}[proof of Corollary~\ref{prod}]
Thanks to Proposition~\ref{aut.prod2}, what remains is to show there is no 
plt-type testconfiguration $((\X',\mathcal{D}'),\mathcal{L}')$ 
which coincides outside ${\rm Supp}(\mathcal{D})$ and 
${\rm Supp}(\mathcal{D}')$ but it follows from 
the presence of polarizations 
since the two spaces are isomorphic in codimension $1$ 
with same restriction of the polarizations in the common open 
locus (cf., \cite{MM}). 
\end{proof}

\vspace{3mm}

Next we characterize the plt-type log test configuration with 
vanishing log Donaldson-Futaki invariant, which is one of 
the keys 
for our later analysis of the stability notion. 

\begin{Prop}[Test configuration as log canonical valuations]\label{tc.lc}
Plt-type test configuration $((\X,\mathcal{D}),\mathcal{L})$ 
for a polarized log Calabi-Yau variety $((X,D),L)$ 
has vanishing log Donaldson-Futaki invariant if and only if 
the discrepancy vanishes 
\begin{align}\label{ld.van}
a_{((X,D)\times \PP^{1})}(\X_{0})=0.
\end{align}
Note that this in particular implies 
the central fiber $\X_{0}$ is log canonical place 
of $((X\times \PP^1, D\times \PP^{1}+X\times \{0\})$ i.e., 
the log discrepancy over $((X\times \PP^1, D\times \PP^{1}+X\times \{0\})$ vanishes. 
Furthermore, such log test configuration is determined by 
the log canonical place as valuation $v_{\X_{0}}$. 
\end{Prop}
\begin{proof}
The former claim follows straightforward from 
the intersection number formula of (log) Donaldson-Futaki invariant 
\cite[Theorem 3.7]{OS}. The latter claim follows from 
that the plt type test configuration has irreducible 
central fiber so that $v_{\X_{0}}$ determines 
$\X$ up to isomorphism in codimension $1$ and the same argument 
as previous Corollary \ref{prod} (\cite{MM}) 
using the polarization. 
\end{proof}

The following supports the terminology. 

\begin{Prop}\label{open.st2}
Weakly open K-polystability (resp., weakly open K-stability) 
of polarized log Calabi-Yau pair $((X,D),L)$ only depends on 
the open locus 
$(X^{o},L^{o})$, not on the choice of compactification $((X,D),L).$

\end{Prop}

\begin{proof}
For $X^o$, suppose there is a triple $((X_1,D_1),L_1)$ which satisfies the 
weak open K-polystability condition as Definition \ref{ops}. 
It suffices to show the following; take another arbitrary polarized dlt Calabi-Yau compactification $((X_2,D_2),L_2)$ 
and we prove it also satisfies the condition of \ref{ops}. 
Consider a log test configuration of $((X_2,D_2),L_2)$ of plt type, with vanishing log Donaldson-Futaki invariant as \ref{ops}, 
whose total space is denoted as $((\X_2,\mathcal{D}_2),\mathcal{L}_2)$ 
and its open subset $\X_{2}^{o}:=\X_2\setminus {\rm Supp}\lfloor \mathcal{D}_2 \rfloor$. 
Take an a priori smaller open subset $\X_{2}^{oo} (\X_{2}^{o})$ from which 
the birational map to $X^o \times \PP^1$ is a morphism. 
We take $(\X'_{1},\mathcal{D}'_1):= ((X_1,D_1)\times (\PP^1\setminus \{0\}))\cup \X_2^{oo}$ 
on which there is also a naturally glued polarization $\mathcal{L}'_{1}$. 
We compactify the triple $((\X'_{1},\mathcal{D}'_1),\mathcal{L}'_{1})$ 
to a log test configuration $((\X'',\mathcal{D}''_1),\mathcal{L}''_{1})$ 
so that $(\X''_1,\mathcal{D}''_1)$ is $\mathbb{Q}$-factorial dlt pair 
with birational {\it morphism} $f$ to $X_1\times \mathbb{P}^1$, by the use of relative minimal model program 
over $X_1\times \mathbb{P}^1$. 
By applying Lemma \ref{tc.lc} to $\X_2$, and combining with the log canonicity of $((X_1 \times \PP^1,D_1 \times \PP^1 +X_1\times \{0\})$, 
we see that there is a $\Q$-divisor $F$ supported on the central fiber $(\X''_1)|_0$ such that 
$(\X''_1,F)$ is log crepant to $((X_1 \times \PP^1,D_1 \times \PP^1 +X_1\times \{0\})$. 
We denote a birational map $\X''_1 \dashrightarrow \X_2$ as $g$. 
We take a small enough $0<\epsilon\ll 1$ 
and run again the relative minimal model  program, this time over $\mathbb{P}^1$, from 
$(\X''_1,F-\epsilon g^{-1}_* (\X_2)|_0)$. Then we obtain a 
dlt minimal model $(\X'''_1, \mathcal{D}'''_1 + (\X'''_1)|_{0})$ 
with generic fiber $(X_1,D_1)$. We take a general relative section of 
$|m L_1 \times \mathbb{P}^1|$ and its closure in $\X'''_1$. Take lc model of 
$(\X'''_1, \mathcal{D}'''_1+\epsilon' A)$ for $0<\epsilon'\ll 1$, then 
by the negativity lemma (cf., e.g., \cite[3.39]{KM}), 
we obtain 
the desired log test configuration $(\X_1, \mathcal{D}_1)$ with the polarization 
extending that of $L_1 \times  (\mathbb{P}^1\setminus \{0\})$. 
Its log Donaldson-Futaki invariant vanishes due to Lemma \ref{tc.lc}, 
hence by the weak open K-polystability (with respect to $((X_1,D_1),L_1)$) assumption we see that this is obtained by 
one parameter subgroup of the automorphism group discussed in Proposition \ref{aut.prod2}. 
By the same proposition, this also provides log product test configuration of $((X_2,D_2),L_2)$ 
which has small birational map to $((\X_2,\mathcal{D}_2),\mathcal{L}_2)$. Applying 
\cite{MM}, we obtain it is isomorphism. Hence we conclude the proof. 
\end{proof}

From the above lemma \ref{tc.lc}, the theory of 
log canonical centers (\cite[4.7, 4.8]{Amb}, \cite[\S 9]{Fjn}) 
play an essential role for this weakly open K-polystability notion. 
See the arguments in the next section \S\ref{ex.sec}. 

The following is motivated by the Shokurov connectedness principle 
and the above Proposition \ref{tc.lc}. 
Take an arbitrary polarized (dlt) log Calabi-Yau pair 
$((X,D),L)$ (Definition \ref{CY.notation} \eqref{notation1}). 
For two log test configurations of plt-type 
$((\X_{i},\mathcal{D}_{i}),\mathcal{L}_{i}) (i=1,2)$, we call one is 
{\it elementary transform} of the other if 
there is a test configuration of $X$ which 
is the blow up of a log canonical center (with reduced structure)
of $(\X_{i},\mathcal{D}_{i}+\X_{i,0})$ for both $i$. 

\begin{Prob}[Connectedness]
For any polarized (dlt) log Calabi-Yau pair 
$((X,D),L)$ and any two log test configurations of plt-type 
$((\X_{i},\mathcal{D}_{i}),\mathcal{L}_{i}) (i=1,2)$ 
with ${\rm DF}((\X_{i},\mathcal{D}_{i}),\mathcal{L}_{i})=0$ for 
$i=1,2$, is there always a finite sequence of such 
plt type log test configurations 
$((\X'_{i},\mathcal{D}'_{i}),\mathcal{L}'_{i}) (i=1,2,\cdots,m)$ 
such that the following holds? 

$(\X'_{i},\mathcal{D}'_{i})$ and $(\X'_{i+1},\mathcal{D}'_{i+1})$ 
are elementary transforms of each other for $i=1,\cdots,m-1$, 
with 
$$(\X_{1},\mathcal{D}_{1})=(\X'_{1},\mathcal{D}'_{1}),$$
$$(\X_{2},\mathcal{D}_{2})=(\X'_{m},\mathcal{D}'_{m}).$$
\end{Prob}

\vspace{5mm}
Finally, as example, we give a systematic construction of 
plt-type test configurations which we use in the next section. 

\begin{Lem}\label{contract}
If $((X,D),L)$ is weakly open K-polystable and 
some irreducible component $D_{i}$ of $\lfloor D\rfloor $ 
satisfies that 
\begin{itemize}
\item 
$L|_{X\setminus D_{i}}\sim_{\Q} \mathcal{O}$ 
and 
\item $D_{i}$ is ample (e.g. when 
$\rho(X)=1$),
\end{itemize}
 then there is a 
 plt-type log test configuration of $((X,D),L)$ such that the 
 central fiber is the projective cone of 
 $((D_{i},\cup_{j\neq i}D_{j}\cap D_{i}),N_{D_{i}/X})$. 
 \end{Lem}

\begin{proof}
We first blow up $D_{i}\times \{0\}\subset X\times \A^{1}$ 
to obtain 
$b \colon \mathcal{B}\to (X\times \A^{1})$ 
whose central fiber is $b^{-1}_{*}X \cup \PP(N_{D_{i}/X}\oplus 
\mathcal{O}_{D_{i}})$. The intersection is still 
$b^{-1}_{*}X \cap \PP(N_{D_{i}/X}\oplus 
\mathcal{O})\simeq D_{i}$. As the component 
$\PP(N_{D_{i}/X}\oplus 
\mathcal{O})$ is exceptional divisor, we simply denote it as $E$. 
Then from the assumption it is easy to see that 
$b^{*}(L\times \A^{1})(-cE)$ gives a contraction to 
a polarized test configuration $\X$ whose central fiber 
$\X_{0}$ is the contraction of $\PP(N_{D_{i}/X}\oplus 
\mathcal{O})$ along the $0$-section, i.e., 
the projective cone of $((D_{i},\cup_{j\neq i}D_{j}\cap D_{i}),N_{D_{i}/X})$ 
which is a process of the usual 
minimal model program which preserves dlt condition (cf., \cite{KM}).  Since the fiber $\X_{0}$ is irreducible, dlt condition implies 
that the obtained polarized family is a plt-type log test configuration. 
\end{proof}


Now we are ready to define open K-polystability and 
its further strengthening, following 
Definition~\ref{ops}. 

From here, 
a main idea could be expressed as to study the structure of 
the ends of the complete metrics to expect, 
``virtually'' considering a 
compact (non-canonical) model which ``close the ends'' 
with small angles, which we analyze by algebro-geometric tools.  

\begin{Def}[Open polystability]\label{ops2}
$(X^{o},L^{o})$ is said to be {\it open K-polystable} 
if and only if there is a polarized dlt log Calabi-Yau pair 
compactification $(X^{o},L^{o})\subset ((X,D),L)$ 
in the sense of Definition~\ref{CY.notation} 
which satisfies the followings; 
\begin{enumerate}
\item 
$\lfloor D \rfloor (\equiv -K_{(X,\{D\})}:=-(K_{X}+\{D\}))$ is nef, 
\item 
there is a constant $\epsilon\in (0,1)$ 
such that 
for any $\beta \in [0,\epsilon)$, the log Donaldson-Futaki invariant 
(\cite{Don11, OS}) is non-negative: 

\begin{align}
{\rm DF}((\mathcal{X},\{\mathcal{D}\}+(1-\beta)\lfloor\mathcal{D}\rfloor),\mathcal{L}\bigl(\frac{c}{\beta}\lfloor\mathcal{D}\rfloor\bigr)))\ge 0
\end{align}
for {\it plt-type} log test configuration 
$$((\mathcal{X},\{\mathcal{D}\}+(1-\beta)\lfloor\mathcal{D}\rfloor),\mathcal{L}\bigl(\frac{c}{\beta}\lfloor\mathcal{D}\rfloor\bigr))$$ 
in the sense of Definition
~\ref{ops} 
of $(X,\{D\}+(1-\beta)\lfloor D\rfloor)$ 
and the value attains $0$ if and 
only if $((\mathcal{X},\{\mathcal{D}\}+(1-\beta)\lfloor\mathcal{D}\rfloor),\mathcal{L}(\frac{c}{\beta}\lfloor\mathcal{D}\rfloor))$ is a product log test configuration. 
\end{enumerate}
We call such a compactification itself 
$(X^{o},L^{o})\subset ((X,D),L)$ {\it stabilizing}, 
or simply open K-polystable. 

In many concrete situations, we only consider the cases when $D$ is integral, 
i.e., $D=\lfloor D\rfloor$ so that the description is simpler. 
\end{Def}

For justification of 
the variation $\frac{c}{\beta}\lfloor \mathcal{D} \rfloor$ of polarization in the above definition, 
see the elliptic type spherical (real) surfaces for instance. 
From the definition, the lemma below immediately follows. 

\begin{Lem} 
Suppose $(X^{o},L^{o})$ is open K-polystable and  
take a stabilizing polarized dlt log Calabi-Yau pair 
compactification $(X^{o},L^{o})\subset ((X,D),L)$. Then, 
it implies: 
\begin{enumerate}
\item $(X^{o},L^{o})$ is weakly open K-polystable 
\item \label{balancing}
if further $L=\lfloor D\rfloor$, 
there is a polarized dlt Calabi-Yau pair compactification 
$((X,D),L)$ such that 
for any log product log test configuration 
$((\mathcal{X},\mathcal{D}),\mathcal{L})$ 
satisfies $(\lfloor \mathcal{D}\rfloor)^{\cdot (n+1)}=0.$ 
\end{enumerate}
We call a compactification 
$((X,D),L)$ of $(X,L)$ satisfying the latter condition 
{\it balancing}. 
\end{Lem}

The last variant definition is: 

\begin{Def}[Strong open polystability]\label{ops3}
In the setting of above Definition~\ref{ops2}, 
$(X^{o},L^{o})$ is called {\it strongly open K-polystable} 
if the chosen compactification $((X,D),L)$ further satisfies that 
\begin{enumerate}
\item 
$\lfloor D \rfloor (\equiv -K_{(X,\{D\})}:=-K_{X}-\{D\})$ is nef and 
\item 
for some fixed positive constant $c>0$, 
$$((X,\{D\}+(1-\beta)\lfloor D\rfloor,L\bigl(\frac{c}{\beta}\lfloor D\rfloor\bigr))$$ 
is log K-polystable and any 
$0<\beta\ll 1$. 
\end{enumerate}
We call such compactification 
$((X,D),L)$ in the either way: {\it strongly stabilizing} compactification, 
stable compactification, or simply, 
being strongly open K-polystable. 
\end{Def}
Then it follows straightforward from the definitions that: 
\begin{Prop}
Strongly open K-polystable $(X^{o},L^{o})$ 
is open K-polystable. 
\end{Prop}
On the other hand, \cite{LX} immediately implies the following. 
\begin{Prop}
In the above situation as \ref{ops3}, suppose that open K-polystable 
$(X^{o},L^{o})$ has a stabilizing compactification 
$((X,D),L)$ which is anti-log-canonically polarized (see \cite{CR} for 
related notion) in the sense that there is 
a small enough positive real number $\beta\ll 1$ such that 
$L$ is proportional to $-(K_{X}+\{D\}+(1-\beta)\lfloor D \rfloor).$ 
Then, it is strongly stabilizing compactification, 
so that $(X^{o},L^{o})$ is strongly open K-polystable in particular. 
\end{Prop}

We add a useful remark, in the spirit of \cite{LiSu}. 
\begin{Lem}\label{interm}
In the above situation as \ref{ops3}, if $(X,L(t\lfloor D \rfloor))$ is K-polystable for any $t\gg 1$, 
then $(X^{o},L^{o})$ is strongly open K-polystable. 
\end{Lem}
\begin{proof}
This follows straightforward from the linearlity of the log Donaldson-Futaki invariant 
with respect to the linear change of the coefficient of the 
boundary divisor $D$ (see \cite[3.7]{OS} and \cite{LiSu}).
\end{proof}

\begin{Rem}
In the case when there is a non-negative linear combination 
of $D_{i}$ which is ample, so that $(X,D)$ 
is asymptotically log Fano in the sense of \cite{CR}, 
effective K-stability criteria have already developed much 
so that the above stability should be able to study by using them. 
\end{Rem}

Finally, we remark that, 
specifying an action of algebraic group $G$ on $(X^{o},L^{o})$ or 
$(X,L)$, we can and do naturally define 
the 
{\it $G$-equivariant version}, such as $G$-equivariantly 
(weak) open K-polystable etc, of above stability notions 
mean that we only concern log test configurations with $G$-
action on whole total space with $G$-linearization. 
Then we can show: 

\begin{Lem}\label{equiv}
If $G$ is a connected algebraic group acting on 
open Calabi-Yau polarized variety $(X^{o},L^{o})$, then for 
$G$-equivariant 
weak open K-polystability of $(X^{o},L^{o})$ 
and 
weak open K-polystability of $(X^{o},L^{o})$ 
are equivalent. 
\end{Lem}
\begin{proof}
This follows from Proposition \ref{tc.lc} (also see 
Corollary\ref{prod}) which 
implies all the log test configurations satisfying the condition in 
{\it loc.cit} should hold $G$-action automatically. 
\end{proof}


\subsection{Reducible setup}

Now we discuss stability notions for polarized 
Calabi-Yau varieties $(X,L)$ where $X$ is non-normal. 
More precisely, we use the notion of 
semi-divisorially log terminality (sdlt for short), as a natural 
non-normal or demi-normal version  of 
original divisorial log terminality. See e.g., 
\cite[\S 2.3]{KM}, \cite[1.1]{FjnM} for details. 

\begin{Def}\label{ops4}
Suppose a connected projective scheme  
$X$ has only semi-dlt singularities and 
$K_{X}\equiv 0$, hence equivalently 
$K_{X}\sim_{\Q} 0$ by \cite{FjnM, Gona}. 

We denote its irreducible decomposition as $X=\cup_{i}V_{i}$ 
with the double locus (conductor divisor) as $D_{i}\subset V_{i}$ 
so that $(V_{i},D_{i})$ is log Calabi-Yau dlt pair for each $i$. 
We also consider a polarization i.e., an ample line bundle on $X$. 

\begin{enumerate}
\item \label{gwps} 
Consider all log test configurations 
$((\mathcal{X},\mathcal{D}),\mathcal{L})$ of $((X,D),L)$ 
which satisfy that 
\begin{itemize}
\item $\mathcal{X}$ satisfies Serre's $S_{2}$ condition, 

\item its restriction to the closure of $V_{i}\times 
(\PP^{1}\setminus \{0\})$ is of plt type in the sense of 
Definition~\ref{ops} \eqref{pltt}, 

\item the log Donaldson-Futaki invariant (cf.,  
\cite{Don11}, \cite[\S 3]{OS}) vanishes. 

\end{itemize}
\noindent
The polarized log Calabi-Yau pair 
$((X,D),L)$ is {\it weakly open K-polystable} 
if any such above type log test configuration $((\mathcal{X},\mathcal{D}),\mathcal{L})$ satisfies that 
{\it the klt (open) locus of}  
$(\mathcal{X},\mathcal{D})$ is a log test configuration of product type of the open locus $(X^{o},L^{o})$. 
By applying Lemma~\ref{prod} to all its components, we see that the condition of such product-ness is also equivalent to that 
the restriction of the log test configuration  
to the closure of $V_{i}\times 
(\PP^{1}\setminus \{0\})$ are 
log {\it product} test configurations for every $i$.

\item \label{gops} 
$((X,D),L)$ is {\it open K-polystable} 
if it is weakly open K-polystable and furthermore 
for each $i$, $((V_{i},D_{i}),L|_{V_{i}})$ is stabilizing in the sense of 
Definition~\ref{ops2}. 

\item \label{gsps} 
$((X,D),L)$ is {\it strongly open K-polystable} 
if it is open K-polystable and furthermore 
for each $i$, $((V_{i},D_{i}),L|_{V_{i}})$ 
is strongly 
stabilizing in the sense of 
Definition~\ref{ops3}. 

\end{enumerate}
\end{Def}
Specifying an action of algebraic group $G$ on $((X,D),L)$, 
we can naturally introduce 
the 
{\it $G$-equivariant version} of above stability notions, 
such as $G$-equivariantly 
open K-polystable etc. 
They mean that we only concern log test configurations with $G$-
action on whole $((\mathcal{X},\mathcal{D}),\mathcal{L})$. 

Similarly to Lemma \ref{equiv}, the following holds. 

\begin{Lem}\label{equiv2}
If $G$ is a connected algebraic group acting on 
sdlt log Calabi-Yau polarized variety $((X,D),L)$, then 
$G$-equivariant 
weak open K-polystability of $((X,D),L)$ holds 
if and only if 
weak open K-polystability of $((X,D),L)$ 
holds. 
\end{Lem}
\begin{proof}
``If'' direction is obvious by definition. 
The ``only if'' direction is reduced to Lemma\ref{equiv} as follows: 
suppose we take a log test configuration 
$((\mathcal{X},\mathcal{D}),\mathcal{L})$ satisfying the 
three conditions in Definition~\ref{ops4}\eqref{gwps}. 
Then, if we take the normalization of $\mathcal{X}$, 
the log Donaldson-Futaki invariant 
${\rm DF}((\mathcal{X},\mathcal{D}),\mathcal{L})$ decomposes to 
the contributions of the log test configurations of each 
closures of $V_{i}\times (\PP^{1}\setminus \{0\})$ e.g., 
by the intersection number formula (\cite{OS}). 
Hence, the assertion follows from Lemma\ref{equiv} or 
more directly from Proposition\ref{tc.lc} that each 
components are product test configurations admitting 
extended $G$-action i.e., corresponding to 
a $\C^{*}$-action which commutes with the given $G$-action. 
By the uniqueness of gluing along conductors 
(\cite[\S 5.6]{Kol13}, also cf., \cite{FjnM}), 
the action also extends to the whole 
log test configuration. 
\end{proof}

\begin{Ques}[Component-wise nature?]
For a polarized semi-dlt Calabi-Yau variety $(X=\cup_{i} V_{i}, L)$ 
to be weakly open K-polystable, is it equivalent to the 
weak open K-polystability for {\it all} 
$(V_{i}^{o},L|_{V_{i}^{o}})$? 
Here, 
$V_{i}^{o}$ denotes the open subset of $V_{i}$ as the 
complement of the double locus $V_{i}\cap 
(\cup_{j\neq i}V_{j})$. 
\end{Ques}

We end the section by 
observing that at least partially this is true. 

\begin{Prop}
For a polarized semi-dlt Calabi-Yau variety $(X=\cup_{i} V_{i}, L)$, 
weak open K-polystabilities for {\it all} 
$(V_{i}^{o},L|_{V_{i}^{o}})$ imply that of $(X,L)$. 

Furthermore, in case $V_{i}$ is smooth (factoriality is enough) and 
the dual complex is one dimensional (i.e., there is no 
three distinct $V_{i}, V_{j}, V_{k}$ intersecting), 
the converse also holds. 
\end{Prop}

\begin{proof}
The former statement is obvious from the definition. 
We prove the latter statement by contradiction. 
If there is an index $i$ such that $(V_{i}^{o},L|_{V_{i}^{o}})$ 
is not weakly open K-polystable, 
there is a plt type test configuration 
$(\mathcal{V}_{i},\mathcal{L}_{i})$ of 
$(V_{i},L|_{V_{i}})$ whose log Donaldson-Futaki invariant is zero 
and is dominated by composite of blow up of 
the connected (or equivalently, irreducible) components 
of $V_{i}\cap (\cup_{j\neq i}V_{j})$. In particular, 
the double loci remain isomorphic as original. 

Then we can glue trivial test configurations of 
$(V_{k},L_{k})$ for all $k\neq i$ and 
$(\mathcal{V}_{i},\mathcal{L}_{i})$ by 
\cite[Thm 9.21]{Kol13} (as varieties) and 
\cite[Prop 9.48]{Kol13} (polarizations). 
This contradicts the weak open K-polystability of 
$(X,L)$. 
\end{proof}


\section{Testing known examples}\label{ex.sec}

This section shows our analysis of the 
stability notions \ref{ops}, \ref{ops2}, \ref{ops3} for 
various class of examples below, which match 
to known (and unknown) gravitational instantons as well as some 
phenomena observed in examples of moduli. 
Here is the list of the results. 

\begin{Thm}\label{ops.ex}
Suppose $((X,D),L)$ is a 
$n$-dimensional polarized dlt log Calabi-Yau pair, 
and take its klt open locus as $X^{o}:=X^{\rm klt}, L^{o}=L|_{X^{o}}$. 
As in section \ref{stab.sec}, 
we do not assume $D$ to be integral divisor 
but can be real divisor. 

\begin{enumerate}

\item \label{..1}
If $D$ is smooth and Cartier, then 
$(X^{o},L^{o})$ is weakly open K-polystable for any $L$. 

\item \label{.1} 
More generally, 
suppose $(X,D)$ is purely log terminal (plt) and further 
$D$ has only canonical singularities and satisfies 
adjunction i.e., $K_{D}=0$ (equivalently, the 
{\it different} in Shokurov's sense is trivial). 
Then it is weakly open K-polystable for any $L$. 

\item \label{.1app}
Under the situation \eqref{.1} above, 
if further one can take $L$ as $-K_{X}$ (hence, with 
$\mathbb{Q}$-Fano $X$), then $(X^{o},L^{o})$ 
is strongly open K-polystable. 

\item \label{.1app2}
Under the situation \eqref{.1} above, 
if we can take the compactification 
$X$ as a semi-Fano manifold 
in the sense of 
\cite[4.11]{ACyl3}, 
then $(X^{o},L^{o})$ 
is again strongly open K-polystable. 

\vspace{2mm}

\item \label{toric} 
If $X^{o}$ is algebraic torus, then $(X^{o},L^{o})$ 
is strongly 
open K-polystable.

\item \label{semiav} 
More generally, if $X^{o}$ is semi-abelian variety, 
then $(X^{o},L^{o})$ 
is strongly open K-polystable. 

\vspace{2mm}

\item \label{logcm}
If $G$ is a reductive algebraic subgroup of 
${\rm Aut}^{o}(X)$ and 
$((X,D),L)$ is weakly open K-polystable, then 
$D$ is GIT polystable with respect to the 
$G$-action. 

\vspace{3mm}

\item \label{cluster}
If $(X,D)$ is a 
cluster log surface (\cite{FG.cluster, GHK}), then 
$((X,D),L)$ is not even weakly open K-polystable for any $L$. 

\item \label{res}
If $X$ is a rational elliptic surface with $D$ a 
nodal fiber of $I_{\nu} (\nu \ge 1)$ type, 
then $((X,D),L)$ is strongly open K-polystable at least for some $L$. 

\item \label{n.}
If $((X,D),L)$ is weakly open K-polystable and 
some irreducible component $D_{i}$ of $\lfloor D\rfloor $ 
satisfies that 
\begin{itemize}
\item 
$L|_{X\setminus D_{i}}\sim_{\Q} 0$ 
and \item $D_{i}$ is ample (e.g. when 
$\rho(X)=1$),
\end{itemize}
 then $((X,D),L)|_{X\setminus D_{i}}$ 
 is the affine 
cone of a certain $(n-1)$-dimensional 
dlt log Calabi-Yau pair $((X',D'),L'):=
((D_{i},\cup_{j\neq i}D_{j}\cap D_{i}),N_{D_{i}/X}).$

\vspace{2mm}

\item \label{ruled}
Suppose $(X,\lfloor D \rfloor = D_{1}\sqcup D_{2})$ with Cartier 
$\lfloor D \rfloor$ and connected $D_{i}$, which 
dominates birationally 
$X'$ which is a $\mathbb{P}^{1}$-bundle over a $(n-1)$-dimensional 
projective variety, such that the images $D'_{i}$ 
of $D_{i}$ are its two sections. 
Then the following holds. 
\begin{itemize}
\item $(D'_{i},{\rm Diff}_{D'_{i}}(0)) (i=1,2)$, 
where ${\rm Diff}(-)$ denotes the Shokurov different 
(cf., e.g., \cite{Kol13}), is 
a klt log Calabi-Yau pair, canonically isomorphic to each other. 
We denote it as $(B,D_{B})$. 
\item If $((X,D),L)$ is weakly open K-polystable for a 
polarization $L$, then $X\to X'$ is isomorphism and 
there is a holomorphic line bundle $N$ on $B$ 
such that $X'=\PP_{B}(\OO\oplus N)$ and 
$D'_{1}$ and $D'_{2}$ are the natural sections, 
the $0$-section and the infinity section 
with respect to the splitting. 

\item Suppose $B$ is an elliptic curve. 
\footnote{We expect this dimension condition would be able to 
removed if we do more refined discussion.}
If $(X^{o},L^{o})$ is further open K-polystable 
as $((X,D),L)$ is stabilizing, 
if and only if $N$ is a numerically trivial line bundle. 
\end{itemize}
\end{enumerate}

\end{Thm}

\begin{Rem}
For \eqref{.1}, 
recall that if $D$ is 
not necessarily Cartier, then in general 
klt condition of $X^{o}, D$ alone 
are not enough to imply plt condition of $(X,D)$ because of the 
failure of adjunction (yielding nontrivial Different 
in the sense of Shokurov). 
\end{Rem}

\begin{Rem}\label{2compo}
For \eqref{.1} again, 
\cite[12.3.2]{FA}, \cite[2.1]{Fjn} 
show that $\lfloor D \rfloor$ has at most two connected 
component and if there are two connected component $D_{1}$ and 
$D_{2}$, they are birational through a $\PP^{1}$-fibration 
over some $(n-1)$-dimensional log terminal base $B$, 
up to a log crepant birational transform.  
We expect that the only open K-polystable such 
$(X^{o},L^{o})$ have a structure $\PP_{B}(\mathcal{O}
\oplus M)$ where $B\in {\rm Pic}^{0}(B)$. 
\end{Rem}

\begin{Rem}
{\it Cluster log surface} in \eqref{cluster} simply 
means the following in our paper, as a simple variant of 
$2$-dimensional 
cluster varieties introduced originally in  
\cite{FG.cluster} (see also \cite{GHK, MV}). 

Let us start from another log smooth Calabi-Yau surface $(X',D')$. 
Then we can blow up a smooth point $p$ of $\lfloor D' \rfloor$, 
which we denote as $\varphi$ as a birational morphism here, 
and take the strict transform of $D'$ to obtain a new 
log smooth Calabi-Yau surface $(X={\rm Bl}_{p}(X'),D=\varphi^{-1}_{*}
D')$. 
In general, cluster log surface in our paper means a 
log smooth Calabi-Yau surface obtained 
by applying this procedure finite times (at least once) 
from a toric log Calabi-Yau pair. 

Also, the statements and the proof of \eqref{cluster} should be easily 
extended to its higher dimensional analogue log Calabi-Yau varieties of type 
\cite{AG20} in which much more deep analysis for the mirror symmetry is done. 
We wish to leave the details to some readers. 
\end{Rem}

\begin{proof}[Proof of Theorem~\ref{ops.ex}]
We provide proofs to each item above one by one. 

\vspace{2mm}

Proof of \eqref{.1}: 
The Iitaka dimension of $D$ is $0$ because of the canonicity of the 
singularity, and therefore in particular it is not uniruled. 
Consider plt-type testconfiguration with vanishing Donaldson-Futaki invariant 
$\X$ then 
the strict transform of $X\times \{0\}$ is log canonical place 
of $(\X_{0},\mathcal{D}_{\X_{0}})$ as Proposition \ref{tc.lc} shows. 
On the other hand, the adjunction says $\X_{0}$ is klt. 
Therefore, \cite{HM} 
implies $D$ with $\kappa(D)=0$ 
cannot be contracted to dimension less than $n-1$ 
in $\X_{0}$. From the fiber connectedness of birational 
morphisms between normal varieties (Zariski's main theorem), 
the center of $\X_{0}$ is birational to $D$. 
Consider the $1$-dimensional general fiber $F$ of the birational map 
$\X_{0}\dashrightarrow D$, the genus $g(F)$ must be 
$0$ because of the Iitaka conjecture (this case of relative dimension $1$ is certainly a theorem cf., e.g., \cite{Kawa}). 

\vspace{2mm}
Proof of Fano manifolds case \eqref{.1app}: 
This follows from \eqref{.1} and 
\cite[5.5]{OS}, 
\cite[2.7]{Od12} (cf., also preceding \cite{Berman}), 
which shows log K-stability for small angle with 
effective bound in terms of the alpha invariant. 

\vspace{2mm}
Proof of semi-Fano manifolds case \eqref{.1app2}: 
if we apply the Kawamata basepointfree theorem to 
$X$, as in \cite{ACyl3}, 
then it reduces to \eqref{.1app} case. 

\vspace{3mm}
Proof of \eqref{toric} and \eqref{semiav}: 
we first prove the weak open K-polystability. 
As preparation, we show the following lemma which 
should be fundamental 
and may be 
known to some experts but the author 
could not find the literatures. We write for 
convenience of readers: 
\begin{Lem}[Torus invariance lemma]\label{tor.val}
Consider an arbitrary toric log Calabi-Yau pair $(V,\Delta_{V})$ 
i.e., toric variety $V$ with the sum $\Delta_{V}$ of all torus invariant 
prime divisors. Its divisorial valuation $v$ whose 
center exists inside $V$ 
is log canonical valuation i.e., $A_{(V,\Delta_{V})}(v)=0$ 
if and only if it is torus invariant. 
\end{Lem}
\begin{proof}[proof of Lemma\ref{tor.val}]
Suppose $v$ is realized as a prime divisor $F$ inside a 
blow up of $V$ as $b\colon W\to V$. 
If $v$ is toric valuation i.e., torus invariant, 
then we can take $W$ and $b$ inside the category of toric varieties. 
If we write the sum of torus invariant prime divisor of $W$ as 
$\Delta_{W}$, then obviously $(W,\Delta_{W})\to (V,\Delta_{V})$ 
is log crepant so that $F$, which is a component of $\Delta_{W}$, 
is obviously a log canonical valuation. 

Conversely, suppose $F$ gives a log canonical valuation $v=v_{F}$. 
Without loss of generality, by toric log resolution again, 
we can and do assume $(V,\Delta_{V})$ is dlt or even log smooth. 
From a lemma of Zariski (cf., \cite[Lemma 2.45]{KM}), 
then $W'$ (which could be a priori non-toric) can be obtained as 
a composition of blow up along the log canonical center of $v$. 
From \cite[4.7, 4.8]{Amb}, \cite[\S 9]{Fjn}, 
it follows that the log canonical center (of $v$) 
in $V$ is a torus invariant strata. Therefore, the 
above obtained $W'$ and the morphism $b'\colon W'\to V$ 
is again toric and $F\subset W'$ realizing $v$ is 
toric. Hence we complete the proof of Lemma\ref{tor.val}. 
\end{proof}

To show the weak open K-polystability of an algebraic torus 
$X^{o}=T$
(there is no amiguity of $L^{o}$ since ${\rm Pic}(X^{o})=0$), 
we use Proposition \ref{tc.lc} again. We take an 
arbitrary toric compactification $((X,D),L)\supset (X^{o},L^{o})$ 
and consider a plt type log test configuration $\X$ such that 
$\X_{0}$ gives 
a log canonical valuation of $(X,D)\times \PP^{1}$. 
Proposition\ref{tc.lc} says it is enough to see that 
$(\X,\overline{D\times (\PP^{1}\setminus \{0\})})$ 
is a product test configuration. 
Lemma \ref{tor.val} shows $v_{\X_{0}}$ is toric, 
while toric valuation is parametrized by 
${\rm Hom}(\mathbb{G}_{m},T\times \mathbb{G}_{m})\otimes \Q$ 
i.e., there is a product test configuration $\X'$ such that 
$\X'_{0}$ also gives $v_{\X_{0}}$. From \cite{MM} again, 
we conclude that $\X'\simeq \X$ hence the proof of 
weak open K-polystability of (polarized) algebraic torus is 
completed. 
The weak open K-polystability of polarized semiabelian 
variety is basically the same since it is \'etale locally 
product of an algebraic torus and smooth base. 

Strong open K-polystability of $X^{o}$ follows from that 
semi-abelian variety $X^{o}$ has a unique\footnote{thanks to  
algebraicity of $A$}
short exact sequence structure 
$1\to T\to X^{o}\to A\to 1$ where $T$ is an algebraic torus, 
$A$ is an abelian variety and if we see this as principal 
$(\C^{*})^{r}(\simeq T(\C))$-bundle, it has flat 
unitary connection or equivalently it is unitary local system. 
Therefore, we can compactify naturally to $X\to A$ as $\PP^{n}$-fiber 
bundle with transition function locally constant. 
That projective bundle corresponds to polystable bundle, and 
hence from Lemma~\ref{interm}, we end the desired proof of \eqref{semiav}. 

\vspace{3mm}
Proof of \eqref{logcm}: We fix $X$ and 
consider the natural universal family of 
$((X,D),L)$ where only $D$ deforms as all $D\in |-K_{X}|$, 
which we denote as $\mathcal{U}_{X}\twoheadrightarrow 
B_{X}=\PP(H^{0}(-K_{X}))$. 
Then consider log CM line bundle $\lambda_{CM}$ on $B_{X}$ 
(cf., e.g., \cite{PT}, \cite[\S 2.4]{ADL}) 
which admits the natural $G$-linearization. 
Since we assume $(X,D)$ is dlt, hence a log-canonical pair, 
Theorem\ref{OSrev} (from \cite{OS}) applies to conclude 
that it is log K-semistable. 
Note that the log CM line bundle 
is of the form $\mathcal{O}_{\PP(H^{0}(-K_{X}))}(c)$ 
but as in the argument  \cite[\S 3.2]{OSS}, \cite[2.22]{ADL}, 
we have $c>0$ and log Donaldson-Futaki invariant of $((X,D),L)$  
is proportional to 
the GIT weight (of $D$) and the proportionality constant is positive. 
Therefore, 
from the arguments 
\cite{PT}, \cite[\S 3.2]{OSS}, \cite[2.22]{ADL}, 
we see that $D$ is GIT semistable with respect to the $G$-action. 
A special case where $X=\PP^{n+1}$ is proved in 
\cite{Lee, Okw}. 

Now, we consider log test configuration 
of $((X,D),L)$  
where the ambient space is a $X$-fiber bundle, in particular, 
of plt type. Since, we now know $D$ is at least semistable 
with respect to $G$-action, weak open K-semistability 
implies the $G$-orbit is closed in $G$-invariant affine 
open subset of $B_{X}$. 
Therefore, we see that weakly open K-polystability of 
 $((X,D),L)$) 
implies the GIT polystability of $D$. 

\vspace{3mm}
Proof of \eqref{cluster}: 
this is easy to see since such a 
 polarized pair $((X,D),L)$ degenerates 
 to a toric Calabi-Yau pair by degenerating the blow up centers to 
 nodes in the toric boundary. Since $X$ has vanishing 
 irregularity, the polarization naturally preserves as well. 
 This forms a product test configuration which is obviously 
 plt-type. On the other hand, since toric polarized Calabi-Yau pair is 
 semi-log-canonical by the presence of toric log resolution 
 (cf., e.g. \cite{Ale1}), \cite[6.3]{OS} or Theorem\ref{OSrev} implies 
 the Donaldson-Futaki invariant vanishes. We conclude the proof of 
 \eqref{cluster}. 

\vspace{3mm}
Proof of \eqref{res}: First we prove weak K-polystability. 
Consider a plt type test configuration $((\X,\mathcal{D}),\mathcal{L})
$ of vanishing Donaldson-Futaki invariant. From \ref{tc.lc}, $\X_{0}$ is a 
log canonical place of 
$(X\times \A^{1},D\times \A^{1}+X\times \{0\})$ 
i.e., $A_{(X\times \A^{1},D\times \A^{1}+X\times \{0\})}(\X_{0})=0$. 
From a lemma due to Zariski (cf., \cite[2.45]{KM}), 
it follows that finite time blow up of $X\times \A^{1}$ at 
the log-canonical center of $\X_{0}$ realizes $\X_{0}$. 
We concretely trace such possibility: 
the irreducible component 
$D_{i}$ of the nodal degenerate fiber $D$ has negative self-
intersection as it is well-known (Hodge index theorem), hence 
we see that the strict transform of the irreducible component 
$D_{i}$ in the strict transform of $X(\times \{0\})$ is still 
negative. We write $X'$ the
other component than $X(\times \{0\})$ 
which includes $D_{i}$. 
From the d-semistability result (a.k.a. the triple point formula) 
we easily see that $D_{i}\subset X'$ has positive self intersection, 
hence not contractible. This would contradict unless 
the log canonical center of $\X_{0}$ is $X\times \{0\}$, hence 
weak K-polystability of $(X^{o},L^{o})$ follows.

Now we prove the strong open K-polystability. 
We regard $X$ as a blow up (with the morphism $\varphi\colon X\to X'$ and the exceptional $(-1)$-curve $E$) of DelPezzo surface $X'$ of 
degree $1$ at the base point of elliptic anticanonical linear system which we denote as $\pi$. 
Then, applying \cite{AP}, it follows that 
$X$ admits a polarization $L_{\epsilon}:=\varphi^*(-K_X')-\epsilon E$ for $\epsilon \ll 1$ 
such that $(X,L)$ has a corresponding 
cscK metric hence K-polystable by \cite{Stp}. 
Therefore, from Lemma~\ref{interm}, we obtain the 
strong open K-polystability of $(X^{o},L^{o})$. 

\vspace{3mm}
Proof of \eqref{n.} follows from Lemma~\ref{contract}. 
We would like to call this pair as being type E in \cite{Od20b}, 
to which we consult the details. 

\vspace{2mm}
Proof of \eqref{ruled}: 
The first item should be 
essentially known to birational geometry experts. 
From a classification result 
(cf., e.g., \cite[2.1]{FjnM}, \cite[5.3]{Gona}), 
it follows that $X'$ is indeed a $\mathbb{P}^{1}$-bundle over 
a klt log Calabi-Yau pair $(D_{i},{\rm Diff}_{D_{i}}(0))$. 
Since $(X,D)$ is a dlt log Calabi-Yau pair, as we assume 
always during Theorem\ref{ops.ex}, 
$(X,D)$ is plt around $D_{i}$. From Cartier assumption of 
$D_{i}$, if we consider the deformation to the normal 
cone with respect to $D_{i}\subset X$, we can contract 
the strict transform of the original central fiber 
$X\times \{0\}$ to get a log test configuration of 
$(X,D)$ degenerating to $\PP_{D_{i}}(\OO\oplus N_{D_{i}/X})$, 
whose log Donaldson-Futaki invariant vanishes. 
Therefore, the second statement on the weak open 
K-polystable case holds if we put $N=N_{D'_{i}/X'}$. 

Finally, for the third statement, 
we suppose $((X,D),L)$ is a stabilizing compactification. 
Then from \cite[5.23]{RT.DG} (also \cite{AT, AK}), 
it follows that $(X',L')$ has vanishing Futaki character 
if and only if $N$ is numerically trivial. 
\end{proof}

For the last claim on \eqref{ruled}, 
we have not been able to succeed to remove dimension assumption 
on $B$ at this point, but we nevertheless expect following 
more general claim would be algebraically proved eventually; 
connecting the phenomenon observed in birational geometry 
(\cite{FA, FjnM, Gona} etc) and the Cheeger-Gromoll splitting 
\cite{CG}. 

\begin{Conj}[Algebraic Cheeger-Gromoll]
If $((X,D),L)$ is open K-polystable and there are two 
connected components of ${\rm Supp}(\lfloor D \rfloor)$ 
(i.e., ``have two ends''), 
then there is a klt log Calabi-Yau variety $(B,D_{B})$ and 
a numerically trivial line bundle $N$ such that 
$X\simeq \mathbb{P}(\OO\oplus N)$, 
$D$ is the union of two natural sections, 
and a complete Ricci-flat weak K\"ahler metric on $X^{o}$ 
is complex analytically locally a product of lower dimension Ricci-flat 
weak K\"ahler metric and a flat metric. 
\end{Conj}
Conversely, if $(B,D_{B}=0)$ is smooth, 
then regarding $X^{o}$ as a unitary local system of rank $1$ 
modulo the natural finite base change, we see the 
existence of complete Ricci-flat K\"ahler metrics on $X^{o}$ in 
the above form. 

\begin{Cor}[Classification for smooth surface case]
\label{surface.classify}
Open K-polystable (resp., weakly open K-polystable) 
smooth surfaces 
$X^{o}$ compactifiable in log smooth $(X,D)$ 
are classified as follows: 
\begin{enumerate}
\item $X^{o}$ is the $2$-dimensionasl algebraic torus 
(i.e., $(X,D)$ is a toric pair), or 
\item \label{ell.ruled}
the complement of an elliptic curve in $X$ or 
\item rank $1$ local system over elliptic curves 
(resp., holomorphic principle $\C^{*}$-bundle over elliptic 
curves). 
\end{enumerate}
\end{Cor}

\begin{proof}
Recall that $D$ is connected or otherwise 
consists of two isomorphic disjoint sections of a ruled surface 
structure (Remark \ref{2compo}). 
The latter case is reduced to above \eqref{ell.ruled} 
which is easy to show. 

After log resolution, we can assume $(X,D)$ is 
either of type \eqref{..1}, i.e., $D$ is smooth, or 
$D$ is nodal i.e., so-called Looijenga pair 
first studied systematically in \cite{Looij}. 
Now, the assertion 
follows from Theorem\ref{ops.ex} \eqref{toric}, \eqref{cluster} 
since \cite[Proposition 1.3]{GHK} shows that 
after finite times composition of blow ups of nodes of the 
boundary curve $D$, 
the surface admits cluster surface structure. 
\end{proof}

\vspace{5mm}
From now on, we examine some known 
gravitational instantons and show their presence (and some 
non-existence) match to 
above Theorem~\ref{ops.ex}. 

\begin{Ex}
For \eqref{res}, this matches to the H-J.Hein's gravitational instanton  \cite{Hajo} 
as well as recent result of \cite[5.22, 6.2]{CJL} 
which relates it with 
Tian-Yau metric \cite{TY} via hyperK\"ahler rotation. Also recall 
G.Chen 
(\cite[Theorem 1.5]{GChen} or \cite[1.2]{CCI})
shows that ALG space $X^{o}$ with curvature decay 
{\it faster than 
quadratic order} can be compactified to a rational elliptic surface 
$X$ and its complement $D$ is its 
singular nodal fiber. In particular, we expect our \eqref{res} 
holds for more general polarization $L$. 
\end{Ex}

\begin{Ex}
An easy typical example of \eqref{logcm} is when $X$ is 
the projective plane $\PP^{2}$. Then weakly open K-polystability of 
$(X,D)$ implies that $D$ is either smooth or union of three 
mutually transversal lines which fits to \ref{..1} and \ref{toric},  
but can not be the union of conic and line. 
\end{Ex}

\begin{Ex}
\cite{TY} matches to above 
Theorem \ref{ops.ex} \eqref{.1app} well. 
Indeed, recall that 
smooth $X$ with smooth $D$ form the most typical plt pair. 
In the construction of Tian-Yau \cite[4.1]{TY} of complete Ricci-flat K\"ahler metric 
on $X^{o}$, they assume  
$D$ is (almost) ample i.e., $X$ is Fano, as the construction 
starts with positive curvature hermitian metric on 
$\mathcal{O}_{X}(D)$ as a source of reference metric.

\end{Ex}

\begin{Ex}
\cite{ACyln} treats the case when 
$X$ admits certain kind of cyclic quotient singularity 
while $D(\supset {\rm Sing}(X))$ is also a 
cyclic finite quotient of a Calabi-Yau manifold. 
Since the finite cyclic group action, ``$\iota$'' in 
{\it loc.cit}, 
do not have fixing divisor, 
$(X,D)$ is plt and as they show the adjunction holds in 
this case. Therefore, the existence result \cite[Theorem D]{ACyln} 
of complete Ricci-flat K\"ahler metric of asymptotically 
cyrindrical type (``ACyl'') perfectly fits as example of above 
condition in Theorem \ref{ops.ex} \eqref{.1}. 

As $3$-dimensional special case, there is a work of 
\cite[see especially 2.6, 4.24, 4.25]{ACyl3} 
when $X$ is certain kind of smooth weak Fano manifold, 
which are then applied to construction of compact $G_{2}$ manifolds 
(\cite{G2}). 
\end{Ex}

We end this subsection by algebraically showing the 
following analogue of Matsushima reductivity theorem 
(\cite{Mat.met}) by reducing to another theorem of 
Y.~Matsushima (\cite{Mat.aff}) on homogeneous spaces! 
However, note that the result is partial for now due to 
reductivity assumption of the ambient symmetry ${\rm Aut}(X,L)$. 

\begin{Cor}[Matsushima-type theorem]
Consider a weakly open K-polystable polarized log Calabi-Yau pair $((X,D),L)$. 
If we assume reductivity of 
${\rm Aut}(X,L)$, then 
${\rm Aut}((X,D),L)$ is also reductive. 
\end{Cor}

\begin{proof}
From the previous Theorem\ref{ops.ex} \eqref{logcm}, 
we see that $D\in |-K_{X}|$ is GIT polystable 
with respect to the natural action of ${\rm Aut}(X,L)$. 
Therefore, its  ${\rm Aut}(X,L)$-orbit is closed in 
an affine subset hence so is affine. 
Thus, \cite{Mat.aff} implies the isotropy of 
$D$ is reductive which is nothing but ${\rm Aut}((X,D),L)$. 
\end{proof}


\section{Compactness and stable reduction}\label{st.red.sec}

\subsection{Introductory remarks}
An important piece of general construction of compact or proper 
moduli space is to establish the formal procedure of constructing 
canonical limits of a sequence or a punctured family of 
objects in concern. In algebraic geometry, 
it is formulated as the 
so-called valuative criterion of properness. 
The purpose of this section is to make a progress in 
the case of moduli of polarized Calabi-Yau varieties $(X,L)$s, 
by following similar idea to the classical Jordan-Holder filtrations. 
With differential geometric perspective, another purpose is to also 
understand the bubbling phenomena of maximally degenerating 
Ricci-flat K\"ahler metrics. 

First we observe that 
flops can cause serious non-separatedness of moduli, 
even in polarized setting, if we do not put any 
stability assumption, which also partially motivates our study. 

\begin{Ex}[Flops cause {\bf un-}separatedness, even with 
polarizations, without polystability conditions]
If has been well-known that families of 
Calabi-Yau varieties can flop, which causes unseparatedness 
(non-Hausdorff properties) of moduli. 

Here, we clarify and enhance the meaning 
by seeing such examples even under the presence of polarizations 
for convenience of readers, as the author could not find literature. 

Consider the well-studied family in $2$-dimensional case 
$\mathcal{X}\to \A^{1}_{t}$ as 
$\X_{t}=[xyzw+tF_{4}(x,y,z,w)]\subset \PP^{3}_{x,y,z,w}$, 
for general $F_{4}\in H^{0}(\PP^{3},\mathcal{O}(4))$. 
There are four $A_{1}$-singularities (conifold points) 
$p_{1},\cdots,p_{4}$ 
on $x=y=t=0$ inside $\X_{0}$ 
which is intersection of the projective plane $V_{1}=(t=x=0)$ 
and 
another projective plane $V_{2}=(t=y=0)$. To be precise, 
$p_{i}$s are zeroes of $F_{4}(0,0,z,w)$. 
(More generally, the total space of generically 
smooth hypersurfaces degeneration $\X=\{tF+G=0\}$ is 
$V(F)\cap {\rm Sing}(\X_{0})$. 
If we blow up $V_{1}$ (resp., $V_{2}$) in $\X$, then we obtain a different 
polarized model $\X_{2}$ (resp., $\X_{1}$). 
On the other hand, 
if we blow up $p_{1}$, then we get yet another model 
$\tilde{\X}$ which dominates both $\X_{i}$s. 
The exceptional divisor of $\tilde{\X}\to \X$ 
is $E\simeq \PP^{1}\times \PP^{1}$ and since $\tilde{X}$ is smooth 
and no critical point generically inside $E$, 
a local section $s(t)\colon \Delta\to \tilde{\X}$ converging to a general point in $E$, the image section inside $\X_{i}$ converges to 
$V_{i}(\subset \X_{0})$ as $t\to 0$,  
whose open locus is not isomorphic. In particular, 
if we simply consider the moduli functor of flat families of 
pointed open polarized Calabi-Yau varieties (``$(X^{o},L^{o})$''s), 
it is not separated. 
\end{Ex}

Another motivation for the following compactness type results 
come from the desire to 
understand more local behavior of some special 
maximally degenerating Calabi-Yau 
metrics. 

More precisely, we concern the pointed Gromov-Hausdorff limits of 
Ricci-flat K\"ahler manifolds with minimal non-collapsing rescale in the following sense, somewhat analogous to the notions of 
``regularity scales'' although not quite identical. 
\begin{Def}[Minimal non-collapsing rescale]\label{minimal.noncollapsing}
A polarized flat ($\Q$-Gorenstein) punctured holomorphic family $\pi^*\colon (\mathcal{X}^*,\mathcal{L}^*)\to \Delta^*:=\Delta\setminus \{0\}$ 
of $n$-dimensional polarized klt projective varieties with continuous family of 
K\"ahler metrics $g_t$ on $\mathcal{X}_t$ whose K\"ahler class is $c_1(\mathcal{L}_t)$, 
take a sequence $p_i\in \mathcal{X}_{t_i} (i=1,2,\cdots)$ such that $t_i (\neq 0) \to 0$ for $i\to \infty$. 
Here, $\mathcal{X}_t:=(\pi^*)^{-1}(t)$ and $\mathcal{L}_t:=\mathcal{L}^*|_{\mathcal{X}_t}$ as usual. 
A sequence of real numbers $r_i$ is called {\it minimal non-collapsing order} if 
$(p_i\in \mathcal{X}_{t_i}, r_i g_{t_i})$ has non-collapsing\footnote{in 
the sense of e.g., \cite{DS2}}
 pointed Gromov-Hausdorff limit as Ricci-flat weak (klt) K\"ahler spaces 
but for any $\epsilon_i>0 (i=1,2,\cdots)$ with $\lim_i \epsilon_i=0$, 
$(p_i\in \mathcal{X}_{t_i}, \epsilon_i r_i g_{t_i})$ does not have such limit. 
\end{Def}
What concerns for us is the natural equivalence class of the sequence 
$\{r_{i}\}_{i}$ by $\sim$ where 
$$\{r_{i}\}_{i}\sim \{r'_{i}\}_{i},$$
simply means $\{\log \frac{r'_{i}}{r_{i}}\}_{i}$ 
is bounded from both sides. 
\vspace{2mm}

Particularly, we are interested in the case when $p_i$ moves along a meromorphic section over $\Delta^*$, 
$K_{\mathcal{X}^*/\Delta^*}\sim_{\Q} 0$, and $g_t$ are Ricci-flat K\"ahler. 
For instance, for degeneration of polarized elliptic curves with standard flat metrics, 
the minimal non-collapsing limit is a cylinder $\C/\mathbb{Z}\simeq \C^*$ with a standard flat metric. The rescaling parameter $r_{t}$ simply makes the 
injectivity radius of $r_{t}g_{t}$ bounded. 
Now, we put an expectation which is 
inspired by \cite{DS2}, which in particular 
gives partial confirmation for smooth case, 
other than the metric completeness. 
\begin{Conj}[Existence of minimally non-collapsing limits]\label{min.nc}
In the above situation, there is a unique equivalence class of 
minimal non-collapsing order. Furthermore, for sequences of points 
$p_i\in \mathcal{X}_{t_i}$ with $t_i (\neq 0) \to 0$ which are 
general enough in a certain 
sense, the pointed Gromov-Hausdorff limits are complete 
Ricci-flat (weak) K\"ahler space. 
\end{Conj}
Here, the ``generality'' condition should be necessary since 
for some sequences, we expect e.g., (multi-)Taub-NUT metrics as the occuring limits 
as \cite{HSVZ} for instance. 

We further put an inprecise expectation as it actually 
served as one of the motivations of our open K-polystability notion. 

\begin{Ques}[Special maximal degeneration case]\label{max.degen.lim.conj}
We expect that for certainly special maximally degenerating families 
of polarized Calabi-Yau varieties, 
the minimal non-collapsing 
pointed Gromov-Hausdorff limit in Conjecture \ref{min.nc} 
are nothing but the open strata of 
open polystable dlt model reductions. 
Characterize when this is true. 
\end{Ques}
However we observe the following. 
\begin{CauCon}\label{Caution}
Note that minimal non-collapsing limits do not respect product 
so that the above expectation 
\ref{max.degen.lim.conj} can not be extended to 
general degenerations. 

To provide a simple counterexample, first we take $(\X^{*},\mathcal{L}^{*})\to \Delta^{*}$ which 
satisfies affirmatively the above Question 
\ref{max.degen.lim.conj}. 
Then, if we take a fixed pointed 
polarized $m$-dimensional Calabi-Yau variety 
$s'\in (S,M)$ and consider 
$(s(t),s')\in 
(\X\times S, p_{1}^{*}\mathcal{L}\otimes p_{2}^{*}M)\to \Delta^{*}$, 
then we expect the minimal non-collapsing limit of the 
fiber Ricci-flat K\"ahler metrics on the fiber $\X_{t}\times S$ 
to be the isometric to the metric product of the following two 
metric spaces. One is the minimal non-collapsing pointed 
Gromov-Hausdorff limit of $\X_{t}\ni s(t)$ with respect to the 
original Ricci-flat K\"ahler metrics, 
and the other is the {\it (metric) tangent cone of} $S$ at $s'$ 
rather than $S$ itself. 

Indeed, this is easily verified in the case of 
degenerating polarized abelian varieties as \cite{TGC.II}. 
\end{CauCon}
Furthermore, by the concrete analysis in \cite{TGC.II}, 
even some maximal degenerations of polarized abelian varieties 
can have negative answers to Question \ref{max.degen.lim.conj}. 
However, certainly good ``balancing'' class of 
maximal degenerations do exist, of which the answer to Question \ref{max.degen.lim.conj} are affirmative. We leave the details to 
future papers. 

\subsection{Weak open stable reduction}

We used the notion of the dlt minimal models for flat family of Calabi-Yau varieties in the above conjecture \ref{max.degen.lim.conj}. 
It is obtained the relative minimal model program over the base curve 
(see \cite{Fjn.ss}) 
gives a powerful method of constructing 
degenerate Calabi-Yau varieties still with trivial 
canonical divisor, in a certain sense. 
A good news is that, as the name suggests, the singularities in such 
degeneration are controlled to be mild (semi-dlt), 
but the filling is rather far from unique. As we mentioned, 
\cite[6.3]{OS} (Theorem ~\ref{OSrev}) 
ensures their log K-semistability but never satisfies 
log K-polystability, hence it does not help to obtain uniqueness. 
This section aims to try to fix the problem by 
restricting the class of such occuring minimal degenerations 
further by imposing our 
our new {\it poly-}stability notions in \S \ref{stab.sec}, 
to obtain (unique and canonical) 
stable reduction type results. 

We work in the setting of 
flat proper family over $\Delta\ni 0$, a germ of 
smooth algebraic curve. Nevertheless, 
we expect the same holds for analytic or formal germ 
if we replace the use of semistable MMP \cite{Fjn.ss} 
by its technical extension to formal equicharacteristic 
setting (cf., e.g., \cite{HP, NKX}). 
We often denote $\Delta\setminus \{0\}$ as 
$\Delta^{*}$. The superscript $^{*}$ denotes for 
punctured setting i.e., 
away from central fiber as boundary, 
while the superscript $^{o}$ means 
outside the horizontal 
boundary ${\rm Supp}(\lfloor \mathcal{D} \rfloor)$. 
In this section, to avoid technical difficulties of dealing with 
degenerations of boundary divisors (cf., e.g., \cite{Kol}), 
we suppose $\mathcal{D}$ is reduced i.e., all coefficients being 
$1$. 

\begin{Thm}[Weak open polystable reduction, Type I case]\label{wstr1}
Consider a flat $\Q$-Gorenstein family of open polarized 
Calabi-Yau varieties (recall Definition~\ref{CY.notation} \eqref{oCY}) 
which we denote by 
$(\mathcal{X}^{*,o},\mathcal{L}^{*,o})\twoheadrightarrow  
\Delta^{*}=\Delta\setminus \{0\}$ 
i.e., both horizontaliry and vertically compactifiable to 
$((\mathcal{X},\mathcal{D}),\mathcal{L})
\twoheadrightarrow \Delta$ 
which is a family of log dlt Calabi-Yau varieties such that 
\begin{itemize}
\item All components of $\mathcal{D}$ have coefficients $1$ and 
are horizontal 
i.e., maps to $\Delta$ dominantly, 
\item 
$\mathcal{X}^{*}\setminus {\rm Supp}(\lfloor \mathcal{D}^{*}\rfloor) 
=\mathcal{X}^{*,o}$, 
\item 
$\mathcal{L}^{*}|_{\mathcal{X}^{*,o}}=\mathcal{L}^{*,o}$, 
\item and 
$(\mathcal{X},\mathcal{X}_{0})$ is of plt-type 
i.e., the open 
central fiber pair 
$(\mathcal{X}^{o}_{0},\mathcal{D}-\lfloor \mathcal{D} 
\rfloor |_{\mathcal{X}^{o}_{0}})$ is klt 
(we call ``Type I'' condition\footnote{after the Kulikov-Pinkham-Persson classification of log smooth minimal 
degenerations of K3 surfaces. This roughly means that 
the pair does not degenerate much.}). 
\end{itemize}
Then, possibly after a finite base change of $\Delta\ni 0$, 
we can perform a birational transform along the fiber over $0$, 
to obtain a priori new 
$((\mathcal{X},\mathcal{D}),\mathcal{L})
\twoheadrightarrow \Delta$ 
such that again 
\begin{itemize}
\item All components of $\mathcal{D}$ have coefficients $1$ and 
are horizontal 
i.e., maps to $\Delta$ dominantly, 
\item 
$\mathcal{X}^{*}\setminus {\rm Supp}(\lfloor \mathcal{D}^{*}\rfloor) 
=\mathcal{X}^{*,o}$, 
\item 
$\mathcal{L}^{*}|_{\mathcal{X}^{*,o}}=\mathcal{L}^{*,o}$, 
\end{itemize}
and further that, most importantly,  
\begin{itemize}
\item 
The central fiber $((\X_{0},\mathcal{D}|_{\X_{0}}),\mathcal{L}|_{\X_{0}})$ 
does not admit a plt type log test configuration which is not a 
product log test configuration. 
In particular, 
the open central fiber 
$((\mathcal{X}^{o}_{0},\mathcal{D}-\lfloor \mathcal{D} 
\rfloor |_{\mathcal{X}^{o}_{0}}),
\mathcal{L}|_{\mathcal{X}^{o}_{0}})$ is weakly open 
K-polystable klt open Calabi-Yau polarized variety in the sense of 
Definition~\ref{ops4} (i). 
\end{itemize}

\end{Thm}

A special consequence for when the general fiber is smooth, 
is as follows. 

\begin{Cor}[{A special case of Theorem~\ref{wstr1}}]
If there is a flat family of polarized 
open Calabi-Yau manifolds 
$(\mathcal{X}^{o},\mathcal{L}^{o})\twoheadrightarrow \Delta$, 
we can replace the central fiber after finite base change 
to make it weak open K-polystable. 
\end{Cor}

\begin{proof}[Proof of Theorem~\ref{wstr1}]
By \cite{Fjn.ss}, there is at least such $((\mathcal{X},\mathcal{D}),\mathcal{L})$ 
which satisfies the conditions except for the last one i.e., 
non-existence of non-product plt log test configuration of the central fiber 
is not ensured. 
\begin{Step}
Suppose there is a 
there is a plt-type log test configuration 
$((\mathcal{X}',\mathcal{D}'),\mathcal{L}')$ 
of the central fiber 
$((\mathcal{X}_{0},\mathcal{D}_{0}),\mathcal{L}|_{\mathcal{X}_{0}})$. 
By the arguments of 
Lemma~\ref{equiv}, Lemma \ref{equiv2}, 
we can even assume 
it is ${\rm Aut}^{o}
(\mathcal{X}^{o}_{0},
\mathcal{L}|_{\mathcal{X}^{o}_{0}})$-equivariant. 

Suppose the base curve germ $\Delta\ni 0$ is realized in a smooth 
algebraic curve $C\ni p$. 
Possibly after a finite base change, we want to glue 
``base changed enough'' base curve $C$ of 
$((\mathcal{X},\mathcal{D}),\mathcal{L})$ 
and the base $\mathbb{P}^{1}$ of 
$((\mathcal{X}',\mathcal{D}'),\mathcal{L}')$ 
to obtain a birational transform of 
$((\mathcal{X},\mathcal{D}),\mathcal{L})$ 
only along the fiber over $0$ so that the new fiber over $0$ becomes 
$(\mathcal{X}'_{0},\mathcal{L}'|
_{\mathcal{X}'_{0}})$. 
We provide its details as following Step \ref{Step1}, Step 
\ref{Step2}, and Step \ref{Step3}. 
\end{Step}

\begin{Step}[Indeterminancy resolution]\label{Step1}
We take a certain component of the incidence locus of the two 
Hilbert schemes one of which 
parametrizes the fibers $(X_{t},L_{t}^{\otimes m})$ for uniform 
$m\gg 0$, as 
${\it Hilb}(\mathbb{P}^{N_{X}})$, 
while the other ${\it Hilb}(\mathbb{P}^{N_{D}})$
parametrizes at least 
$(D_{t},\mathcal{L}^{\otimes m}|_{D_{t}})$ for the same fixed $m$. 
We denote such incidence locus by $H$. 
Furthermore, we suppose the destablizing log test configuration 
$((\mathcal{X}',\mathcal{D}'),\mathcal{L}')$ 
 is induced by a 
$\mathbb{C}^{*}$-action $\lambda$ on $\mathbb{P}^{N}$ hence on 
$H$. 

This gives a flat projective family over $\mathbb{C}^{*}\times C$ in a $\mathbb{C}^*$-equivariant manner, extending $\mathcal{X}$ over $\{1\}\times C$, 
hence a 
morphism $\mathbb{C}^{*}\times C\to {\it Hilb}(\mathbb{P}^{N})$. In particular, it induces a rational map 
$$\mathbb{P}^{1}\times C\dashrightarrow H.$$ 
We consider its indeterminancy locus which must be finite closed points inside $\{0,\infty\}\times C$. 
In particular, we can take $\{p=p_1,\cdots, p_l\}\subset C$ so that the indeterminancy locus is inside 
$\{0,\infty\}\times\{p=p_1,\cdots, p_l\}$. 
Since we assume that the order of the base field $\mathtt{k}$ is infinity, 
note that we can replace the subset $\{p=p_1,\cdots, p_l\}$ of $C$ by a larger one with arbitrarily big order if we need, as we do in the next Step \ref{Step2}. 

Now we consider a $\C^*$-equivariant resolution of 
indeterminancy of the rational map which we denote as $\mathcal{B}\to H$, 
which is a blow up of some close subscheme of $\mathbb{P}^1\times C$ which we denote as $\Sigma$. 
Then we obtain a diagram 
\[
\xymatrix{
&  &\mathcal{B}={\it Bl}_{\Sigma}(\mathbb{P}^1\times C)\ar@{->>}[ld]^-\pi \ar[rd]\\ 
& \mathbb{P}^1\times C & & 
H, 
\\ 
}
\]
and our flat projective family (over ($\mathbb{C}^{*}\times C$)) extends over $\mathcal{B}$. 
This is an easy example of so-called flattening procedure. 

Now we define {\it depth} of the blow up $\pi$ at a point in $C\times \mathbb{P}^1$. 
First, since $\pi$ is a birational proper morphism from regular surface to regular surface, it can be decomposed as 
the composition of finite maximal ideal blow up as 
$\pi=\pi^{(d)}\circ \pi^{(d-1)}\circ \cdots \circ \pi^{(1)}$ where each $\pi^{(i)}$ is a blow up 
with its corresponding maximal ideal defining a closed point $c_i$. We define 
${\rm depth}(\pi, (x,y))$ as $$\#\{i\mid c_i=(\pi^{(i-1)}\circ \cdots \circ \pi^{(1)})^{*}(x\times C)\cap 
(\pi^{(i-1)}\circ \cdots \circ \pi^{(1)})^{-1}_{*}(\tilde{C}\times y)\}.$$ 
Note that ${\rm depth}(\pi, (x,y))=0$ unless $x$ is either $0$ or $\infty$ and $y$ is one of $p_i$s. 
Then take a positive integer $m$ such that 
\begin{align}
\label{depth.cond} m> \max_{(x,y)} \{{\rm depth}(\pi, (x,y))\}, 
\end{align}

\end{Step}

\begin{Step}[Cyclic covering]\label{Step2}

If we replace $\{p_1,\cdots,p_l\}$ by a larger set if necessary, 
we can and do assume that there is an effective $\Z$-divisor $D$ on $C$ such that 
$\bigl(\sum_{1\le i \le l}p_i\bigr)\sim mD.$
\footnote{Here we used that the base field is characteristic $0$ hence 
infinity order in particular. }

We set $L:=\mathcal{O}_C(D)$ on $C$, and exploits the standard cyclic cover construction: 
$$\tilde{C}:=\mathit{Spec}_{\mathcal{O}_C}(\oplus_{0\le a<m}  \mathcal{O}_{C}(aD))\twoheadrightarrow C,$$
where the ring structure on the right hand side is induced by 
$\mathcal{O}_C(mD)\simeq \mathcal{O}_C(\sum_i p_i)\hookrightarrow \mathcal{O}_C$. 
Since $m$ is coprime to the characteristic of $\mathtt{k}$, this gives an integral normal curve 
$\tilde{C}$ which is cyclic cover of $C$. We denote the covering by $f$ and write its graph in $\tilde{C}\times C$ as 
$\Gamma_f$. Recall that $\Gamma_f$ is isomorphic to $\tilde{C}$ through the projection. 

Note that $f^{*}p_i=m  p'_i$ with $p'_i\in \tilde{C}$ for each $i$. 

\end{Step}

\begin{Step}[Rational function for base change]\label{Step3}

We take a pair of disjoint finite closed sets $\{q_j\}_j \subset C$ and $\{q'_j\}_j \subset C$ such that: 
\begin{enumerate}
\item $\{q_j\}=\{q'_j\}\ge 2g(\tilde{C})+1, $
\item $\{p'_i\}\cap \{q_j\}=\{p'_1\},$ 
\item $(\{p'_i\}\cup \{q_j\})\cap \{q'_j\}=\emptyset$, 
\item $\sum_j q_j \sim \sum_j q'_j.$ 
\end{enumerate}

From the last condition, 
we have a rational function $r\in K(\tilde{C})^*$ whose zeroes are $\{q_j\}$ with multiplicities $1$ and 
poles $\{q'_j\}$ with multiplicities $1$. In particular, $r$ is \'etale over $\{0,\infty\}\subset \mathbb{P}^1$. 
Therefore, $\tilde{\mathcal{B}}:=\mathcal{B}\times_{(\mathbb{P}^1\times C)} (\tilde{C}\times C)$ is smooth, 
on which we also have a flat polarized family extending that of original $(\mathcal{X},\mathcal{L})$ which is isotrivial for $\tilde{C}$-direction. 
We denote the obtained morphism $\tilde{\mathcal{B}}\to (\tilde{C}\times C)$ by $\tilde{\pi}$, 
which can be also seen as an indeterminancy resolution for the corresponding rational map to ${\it Hilb}(\mathbb{P}^{N})$. 
Summarizing, we have the following diagram: 
\[
\xymatrix{
&                                                          & \Gamma_f \ar@{^{(}-_>}[d]\\ 
&\tilde{\mathcal{B}} \ar@{->>}^{\widetilde{(r\times {\it id})}}[d]\ar@{->>}^{\tilde{\pi}}[r] &(\tilde{C}\times C) \ar@{->>}^{(r\times {\it id})}[d]  \\ 
&\mathcal{B}=Bl_{\Sigma}(\mathbb{P}^1\times C) \ar@{->>}^-{\pi}[r] & (\mathbb{P}^1\times C).
}
\]

Then we consider $\tilde{\pi}^{-1}_* \Gamma_f \subset 
\tilde{\mathcal{B}}$ which is isomorphic to $\tilde{C}$. 
The family over $\Gamma_{f}\simeq \tilde{C}$, 
pulled back from the incidence locus of 
${\rm Hilb}(\PP^{N_{X}})\times {\rm Hilb}(\PP^{N_{D}})$ is our desired 
glued family of polarized log Calabi-Yau varieties. 
\end{Step}

\begin{Step}
We denote the normalization of the new family over $\tilde{C}$ as 
$\mathcal{X}(1)$ and replace the notion $\tilde{C}$ by 
$C$ for simplicity. Note the normalization map of $\mathcal{X}$ 
is bijective, which follows from fiberwise $S_{2}$-condition. 
Therefore, the pullback of the boundary divisor makes sense 
which we denote as $\mathcal{D}(1)$. Furthermore, we 
put the pullback of the polarization denoted as $\mathcal{L}(1)$. 
A priori, $(\mathcal{X}(1),\mathcal{D}(1))$ may be only 
log canonical, not necessarily dlt, by the inversion of adjunction 
(\cite{Kwkt}). Therefore, we consider its dlt modification (cf., e.g., \cite{Kol13, FG, OX}) 
to modify to a dlt minimal model over $\Delta$. 
\end{Step}

\begin{Step}
Notice that the rank of 
${\rm Aut}^{o}
(\mathcal{X}^{o}_{0},
\mathcal{L}|_{\mathcal{X}^{o}_{0}})$ increases by the above procedure, 
hence the replacement either stops after finitely many times, 
to obtain a sequence 
$((\mathcal{X}(i),\mathcal{D}(i)),\mathcal{L}(i)) (i=1,2,\cdots)$ 
or 
$\mathcal{X}^{o}_{0}$ becomes an algebraic torus so that the 
assertion holds by Theorem\ref{ops.ex} \eqref{toric} after all.  
\end{Step}
\end{proof}

\vspace{2mm}
To state a variant for general reducible degenerations, 
we introduce the following a priori variant stability notion, which 
might 
remind the readers of the notion of pointed Gromov-Hausdorff limits 
in metric geometry. However, we recall from Caution~\ref{Caution} that 
the open parts of the components of degenerate 
varieties which satisfy these 
(pointed) open polystability notion can {\it not} be pointed Gromov-
Hausdorff limits themselves in general, especially for non-maximal degenerations, 
even after rescale (cf., e.g., \cite{TGC.II}). 
Neverthelss, weaker relations can be expected in special situations 
and we wish to explore 
relations in more general context in future.

\begin{Def}\label{ops5}
We suppose a projective variety 
$X$ has only sdlt singularities and 
$K_{X}\sim_{\Q} 0$ and the irreducible decomposition as $X=\cup_{i}V_{i}$ 
so that each $(V_{i},D_{i})$ is a log Calabi-Yau dlt pair, 
where $D_{i}$ denotes the conductor divisor. 

Now we fix a reference 
closed point $p$ inside the 
klt locus of $\X_{0}$. 
We also fix a polarization i.e., an ample line bundle on $X$. 

For a log test configuration 
$((\mathcal{X},\mathcal{D}),\mathcal{L})$ of $((X,D),L)$ 
such that 
\begin{itemize}
\item $\mathcal{X}$ satisfies Serre's $S_{2}$ condition. 

\item Its restriction to the closure of $V_{i}\times 
(\PP^{1}\setminus \{0\})$ is of plt type in the sense of 
Definition~\ref{ops} \eqref{pltt}. 

\item The log Donaldson-Futaki invariant 
${\rm DF}((\mathcal{X},\mathcal{D}),\mathcal{L})$ 
(in the sense of 
\cite{Don11, OS}) vanishes. 

\item The limit of $p$ 
i.e., $\overline{\mathbb{G}_{m}\cdot (p\times \{1\})}\cap \mathcal{X}_{0}$
 is inside the klt locus of 
$(\X_{0},\mathcal{D}_{0})$. 

\end{itemize}

$((X,D),L)$ is {\it weakly pointed open K-polystable} 
if every such log test configuration $((\mathcal{X},\mathcal{D}),\mathcal{L})$ satisfies that the klt locus of 
$(\mathcal{X},\mathcal{D})$ is of product type. 
\end{Def}

\begin{Thm}[Weak (proper) pointed stable reduction]\label{wstr2}
Take an arbitrary flat $\mathbb{Q}$-Gorenstein family 
$((\mathcal{X}^{*},\mathcal{D}^{*}),\mathcal{L}^{*})\to \Delta^{*}$ 
of 
polarized projective klt log Calabi-Yau varieties 
over a punctured germ of smooth curve $0\in \Delta$, 
with $K_{\X^{*}/\Delta^{*}}+\mathcal{D}^{*}
\equiv 0$ 
\footnote{again, equivalent to 
$K_{\X^{*}/\Delta^{*}}+\mathcal{D}^{*}\sim_{\Q} 0$ by 
\cite{FjnM, Gona}} 
 and a meromorphic section $s\colon \Delta^{*}
\to \X^{*}$, possibly after a finite base change of $\Delta\ni 0$, 
there is a dlt minimal model 
$(\mathcal{X},\mathcal{D}+\X_{0})\to \Delta$ 
such that the following additional conditions hold: 
\begin{enumerate}
\item $(\X_{0},\mathcal{D}|_{\X_{0}})$ is semi-dlt and 
$K_{\X_{0}}+\mathcal{D}|_{\X_{0}}\sim_{\Q}0$, 
\item \label{insideklt}
$\lim_{t\to 0}s(t)$ lies inside the klt locus $\X_{0}^{\rm klt}$ of 
the central fiber $(\X_{0},\mathcal{D}_{0})$ 
\item \label{wpops}
$((\X_{0},\mathcal{D}|_{\X_{0}}),\mathcal{L}|_{\X_{0}})$ is a weakly pointed 
open K-polystable 
Calabi-Yau polarized variety with respect to $\lim_{t\to 0} s(t)$ 
in the sense of Definition~\ref{ops5}. 
\end{enumerate}
\end{Thm}

Again, we also write brief version for a rather special i.e., 
log smooth case, for readers' convenience. 

\begin{Cor}\label{wstr3}
Take an arbitrary degenerating polarized family $(\mathcal{X},\mathcal{L})
\twoheadrightarrow \Delta$ 
of $n$-dimensional pointed smooth (projective) polarized 
Calabi-Yau manifolds also with a holomorphic section $\{s(t)\}_{t}$. 
We write the fibers as 
$(\X_{t},\mathcal{L}_{t})\ni s(t) (t\neq 0)$ 
and suppose the degeneration $\mathcal{X}_{0}$ 
is a simple normal crossing Calabi-Yau 
variety. 
For instance, polarized Kulikov families of 
pointed K3 surfaces, not of Type I, satisfy the condition. 

Then, we can birationally modify the central fiber, possibly after a finite base change, 
to make it a ``better'' degenerate polarized Calabi-Yau varieties 
$(\X_{0},\mathcal{L}_{0})$ 
which satisfies: 
\begin{enumerate}
\item $\X_{0}$ is 
simple normal crossing away from a closed subset of 
dimension $n-2$, 
\item  the limit point of ${\rm lim}_{t\to 0}s(t)$ 
stays outside the double locus of $\X_{0}$, 
\item $(\X_{0},\mathcal{L}|_{\X_{0}})$ is 
weakly pointed open K-polystable 
with respect to the limit point of ${\rm lim}_{t\to 0}s(t)$.
\end{enumerate}
\end{Cor}

\begin{proof}[proof of Theorem~\ref{wstr2}]
\begin{Stp}
By the semistable reduction theorem \cite[Chapter II, III]{KKMSD} 
and the semistable minimal model program \cite{Fjn.ss}, 
we can at least construct a dlt minimal model 
$(\X,\mathcal{D}+\X_{0})\to \Delta$. 
\end{Stp}

\begin{Stp}\label{b.c.}
The main innovation from here is the following base change trick, 
which more or less yields ``sub-divided'' 
dlt minimal models 
with much more irreducible components in the central fiber. 

Using the $N$-th ramifying finite morphism 
$b^{(N)}\colon (\Delta\ni 0)\to (\Delta\ni 0)$ with 
$N\in \Z_{>0}$, 
we take the base change of $(\X,\mathcal{D})$ 
which we denote as $(\X^{(N)},\mathcal{D}^{(N)})
=\X\times_{\Delta, b^{(N)}}\Delta$. 
Recalling the determination of log canonical centers of dlt pairs 
(cf., \cite{Amb}, 
\cite[\S 3.9]{Fjnlt}), combined with \cite[proof of 5.20]{KM}, the lc 
centers of $(\X^{(N)},\mathcal{D}^{(N)}+\X^{(N)}_{0})$ 
are simply the preimages of 
those of $(\X,\mathcal{D}+\X_{0})$. 
Let us consider the open subset of $\X$ where 
$(\X,\mathcal{D}+\X_{0})$ is log smooth, which we denote as 
$\X^{\rm lsm}$ and think of \'etale local structure of 
$\X^{\rm lsm}\times_{\Delta, b^{(N)}}\Delta$. 
From the log smoothness of $(\X,\mathcal{D}) \cap \X^{\rm lsm}$, 
the preimage of $\X^{\rm lsm}$ in 
$(\X^{(N)},\mathcal{D}^{(N)}+\X^{(N)}_{0})$ are toroidal 
and have simple explicit combinatorial description 
by a rational polyhedral fan over the 
dual intersection complex (\cite[Chapter II, III]{KKMSD}). 
We take an arbitrary regular 
subdivision of the corresponding fan, then it gives a 
crepant toric log resolution of 
the preimage of $\X^{\rm lsm}$ inside 
$\X^{(N)}$. 
\end{Stp}

\begin{Stp}
Suppose that log resolution is given by the blow up of a coherent 
ideal $I^{o}$ in the preimage of $\X^{\rm lsm}$ inside 
$\X^{(N)}$. We extend the coherent ideal $I^{o}$ to a coherent 
ideal $I$ of 
whole $\mathcal{O}_{\X^{(N)}}$ whose normalized blow up gives still 
a log resolution (or we take the log canonical closure of \cite{HX}). 
Then we run the relative minimal model program 
over $\mathcal{X}^{(N)}$, 
which is now allowed by using \cite{HX, HH} etc. This gives a 
dlt minimal model over $\mathcal{X}^{(N)}$ which we denote by 
$(\mathcal{X}^{[N]},\mathcal{D}^{[N]})$. By a simple phenomenon that 
the all log canonical centers of $(\mathcal{X},\mathcal{D})$ intersect with 
$\X^{\rm lsm}$ (cf., \cite{Amb}, 
\cite[\S 3.9]{Fjnlt}), it follows that 
$(\mathcal{X}^{[N]},\mathcal{D}^{[N]})$ is also a dlt minimal model {\it over the base curve} $\Delta$ as well. 

From our earlier determination of log canonical centers of 
$(\X^{(N)},\mathcal{D}^{(N)}+\X^{(N)}_{0})$, it easily follows that 
all the irreducible components of $\X^{[N]}_{0}$ 
intersect with the preimage of $\X^{\rm lsm}$. 
\end{Stp}

\begin{Stp}
Now we consider ${\rm order}_{t=0}(s^{*}V_{i}) \in \Q$ and 
${\rm order}_{t=0}(s^{*}\mathcal{D}_{i})\in \Q$ 
where $V_{i}$ denotes the 
irreducible component of $\X_{0}$ and 
$\mathcal{D}_{i}$ denotes irreducible components of 
$\mathcal{D}$. Comparing with those rational 
numbers, 
if we take a prime number $N$ which do not divide 
neither the numerators or denominators, then our 
desired assertion \eqref{insideklt} 
easily follows. This fact can be 
shown in standard explicit blow up calculation while it more systematically 
follows by use of the Morgan-Shalen-Boucksom-Jonsson construction 
(cf., e.g., \cite[Appendix]{TGC.II}) which we expand 
in details in another paper \cite{mod} more. So now we finish 
the discussion of our base change trick which 
ensures the dlt model existence satisfying the 
condition \eqref{insideklt}. 
\end{Stp}

\begin{Stp}
As a next step, we wish to improve further the obtained model 
$(\X^{(N)},\mathcal{D}^{(N)})$ to make it satisfy \eqref{wpops} i.e., 
the weak pointed open K-polystability as desired. 
Here, we use the same trick as previous Theorem\ref{wstr1} 
of ``gluing in'' test configuration; 
we take an arbitrary ample extension $\mathcal{L}^{(N)}$ 
which extends that of (base change of) $\mathcal{L}^{*}$. 
If $((\X^{(N)}_{0},\mathcal{D}^{(N)}_{0}),\mathcal{L}^{(N)}|
_{\X^{(N)}_{0}})$ 
is not weakly pointed open K-polystable with respect to the 
limit $p:=\lim_{t\to 0}s(t)$ of meromorphic section $s$ in $\X^{(N)}$, 
then there is a log test configuration of 
$((\X^{(N)}_{0},\mathcal{D}^{(N)}_{0}),\mathcal{L}^{(N)}|
_{\X^{(N)}_{0}})\ni p$ 
of plt-type 
$((\X',\mathcal{D}'),\mathcal{L}')$ such that the 
limit of $p$ is inside klt locus of 
$(\X^{(N)}_{0},\mathcal{D}^{(N)}_{0})$, 
the log Donaldson-Futaki invariant vanishes 
${\rm DF}(\X^{(N)}_{0},\mathcal{D}^{(N)}_{0})=0$, and not of product 
type. 
More precisely, completely similarly to 
the Steps \ref{Step1},\ref{Step2},\ref{Step3} of the proof of  
Theorem~\ref{wstr1}, 
we take a sufficiently high degree ramifying base change 
of $(\X^{(N)},\mathcal{D}^{(N)})$ and 
glue them along the fiber over $0\in \Delta$ 
with the above log plt-type test configuration 
$((\X',\mathcal{D}'),\mathcal{L}')$. 
Therefore we obtain $((\mathcal{X},\mathcal{D}),\mathcal{L})$ which 
satisfies \eqref{wpops}. 
\end{Stp}

\begin{Stp}
Finally we want to ensure the dlt property, and we slightly 
modify as follows: we first take (semi-)normalization of 
the total space $\mathcal{X}$ which must be bijective to $\X$, 
then the inversion of adjunction tells us that 
$(\X,\mathcal{D}+\X_{0})$ is log canonical. 
Hence we can take log crepant dlt blow up as in 
e.g., \cite{FG, OX, Kol13}, which is desired model. 
The condition \eqref{insideklt} is automatically 
satisfied because of the 
glueing construction which preserves the same property 
for $((\X',\mathcal{D}'),\mathcal{L}')$. 
\end{Stp}
\end{proof}
We explore the above Step\ref{b.c.} more systematically and 
thoroughly in another paper in preparation. 
If we replace the use of weak {\it pointed} open 
K-polystability in above \ref{wstr2}, \ref{wstr3}, 
by the weak open K-polystability 
in the sense of Definition~\ref{ops4} \eqref{gwps} (unpointed version), 
we also obtain the following. Since the proof goes 
same way as that of Theorem \ref{wstr1} just by replacing the definition of 
stability, hence easier than \ref{wstr2}, we omit the proof. 

\begin{Thm}[Weak (proper unpointed) stable reduction]\label{wstr4}
Take an arbitrary flat $\mathbb{Q}$-Gorenstein family 
$((\mathcal{X}^{*},\mathcal{D}^{*}),\mathcal{L}^{*})\to \Delta^{*}$ 
of 
polarized projective klt log Calabi-Yau varieties 
over a punctured germ of smooth curve $0\in \Delta$, 
with $K_{\X^{*}/\Delta^{*}}+\mathcal{D}^{*}
\equiv 0$ 
\footnote{again, equivalent to 
$K_{\X^{*}/\Delta^{*}}+\mathcal{D}^{*}\sim_{\Q} 0$ by 
\cite{FjnM, Gona}}, 
possibly after a finite base change of $\Delta\ni 0$, 
there is a dlt minimal model 
$(\mathcal{X},\mathcal{D}+\X_{0})\to \Delta$ 
such that the following additional conditions hold: 
\begin{enumerate}
\item $(\X_{0},\mathcal{D}|_{\X_{0}})$ is semi-dlt and 
$K_{\X_{0}}+\mathcal{D}|_{\X_{0}}\sim_{\Q}0$, 
\item \label{wpops}
$((\X_{0},\mathcal{D}|_{\X_{0}}),\mathcal{L}|_{\X_{0}})$ is 
weak open K-polystable 
(in the sense of Definition~\ref{ops4} \eqref{gwps}). 
\end{enumerate}
\end{Thm}

\begin{Rem}\label{GHKrecover}
After above Theorem \ref{ops.ex}, 
Corollary\ref{surface.classify} etc, 
we expect that the observation by Gross-Hacking-Keel \cite{GHK} 
that any Type III one parameter algebraic degeneration can be
vertically birationally transformed into a toric degeneration, 
can be seen as a special case of strong open K-polystable reduction. 
\end{Rem}
In another paper \cite{mod}, 
we discuss a continuation of the above discussion in the proof. 

\begin{Rem}
Consider a log Calabi-Yau lc pair $(\PP^{2}, \frac{3}{4}C)$ 
where $C$ is a ``cat-eye'' i.e., a union of two smooth 
conics $C_{1}$ and $C_{2}$ intersecting at two tacnodes, 
i.e., $1$-dimensional $A_{3}$-singularities (cf., \cite{HL, OSS, ADL}). 
Take the blow up of both tacnodes 
$C_{1}\cap C_{2}$ and further blow up 
the two singular points of the strict transform of $C$. 
Then we obtain a birational ruled surface 
$X\twoheadrightarrow \PP^{1}$ with two sections $D_{1}$ and 
$D_{2}$. We can make this into a log crepant resolution 
$(X,D_{1}+D_{2}+\Delta)\to \mathbb{P}^{2}$. 
This log Calabi-Yau $(X,D_{1}+D_{2}+\Delta)$ is dlt and 
we can easily see its weak open K-polystability by its 
simple structure of log canonical centers by 
Theorem~\ref{tc.lc}. However, in the wall crossing of 
\cite{ADL}, they replace this by 
$(\PP(1,1,4),\frac{3}{4}[x^{4}y^{4}=z^{2}])$ 
(cf., \cite{ADL}, also \cite{OSS}). This 
gives an example of further reduction process. 
See \cite{mod} for further discussion. 
\end{Rem}

\subsection{Strong open stable reduction}
This subsection shows another stable reduction type theorem in 
the case approximated by log Fano pairs with 
conical singular weak K\"ahler-Einstein metrics. We start with a subtle remark. 
\begin{Rem}
Consider a punctured family $((\X^{*}, \mathcal{D}^{*}),\mathcal{L}^{*})\to \Delta^{*}$ 
of strongly open K-polystable log Calabi-Yau pairs, 
whose fibers $\X_{t} (t\neq 0)$ are $\Q$-Fano varieties. 
Even if they are smooth Fano varieties with K\"ahler-Einstein metrics, 
whose existence ensures the strongly open K-polystabilities 
of $((\X_{t},\mathcal{D}_{t}),\mathcal{L}_{t})$ for any $t\neq 0$ 
(cf., \cite{LiSu}, 
\cite{OS}), in general, 
we can{\it -not} get strongly open K-polystable central fiber 
after possible finite base change, 
simply by applying the K-stable reduction to $\X_{t}$s by 
\cite{DS}. 

For instance, take a family of Tian-Yau metric on the complements of 
quartic K3 surfaces which degenerates to the doubled quadric surface. 
\end{Rem}
Nevertheless, by making the conical angles rather small (acute), we still 
get following result. 
\vspace{2mm}
\begin{Thm}[Strong open K-polystable reduction]\label{ststr}
Consider a punctured polarized proper family 
$((\X^{*}, \mathcal{D}^{*}),\mathcal{L}^{*})\to \Delta^{*}=\Delta\setminus 
\{0\}$ 
of strongly open K-polystable log Calabi-Yau pairs. We suppose 
either of the following: 
\begin{enumerate}
\item \label{assum1} the fibres are 
asymptotic log Fano pairs in the sense of 
\cite{CR} and a conjecture on $\delta$-invariant minimization 
(\cite[Conjecture 5.1]{ZZ}) holds. 
\item \label{assum2} simply $X_{t}$s are Fano manifolds 
with conical K\"ahler-Einstein metrics whose singularities $D_{t}$ are smooth. 
\end{enumerate}

Then, possibly after finite base change, 
there is its model $(\X,\mathcal{D}) \twoheadrightarrow \Delta\ni 0$ 
such that the central fiber $((\X_{0},\mathcal{D}_{0}),\mathcal{L}_{0})$ 
is strongly open K-polystable. 
Moreover, such filling is unique. 
\end{Thm}
\vspace{1mm}
\begin{proof}
By the Zariski openness of ampleness of line bundles for flat variation, 
there is a uniform constant $0<c_{0} \ll 1$ such that 
$(\X_{t},c\mathcal{D}_{t})$ are (klt) log Fano pairs for any $(0<)c<c_{0}$ and 
any $t\in \Delta^{*}$. 

By \cite[5.4, 5.5]{OS} and the usual Skoda type estimate 
(cf., e.g., \cite[\S 2]{Od}), there is a uniform $\beta_{0}>0$ 
such that $(\X_{t},(1-\beta)\mathcal{D}_{t})$ is 
log K-polystable for any $(0<)\beta<\beta_{0}$ and any $t$. 

For any such fixed $\beta$, from \cite{BHLLX} for the case \eqref{assum1} 
and from \cite[III, Theorem 2]{CDS} in the case \eqref{assum2} 
(cf., also \cite{Tia}), there is a model $(\X,\mathcal{D})$ such that 
$(\X_{t},(1-\beta)\mathcal{D}_{t})$ is a log K-polystable $\Q$-Fano pair. 

What remains is to show its independence of small $\beta$. 
Note that 
the log canonical threshold of above obtained $\X_{0}$ 
with respect to $\mathcal{D}_{t}$ is at least $1-\beta$ 
from the klt property of the obtained pair 
$(\X_{0},(1-\beta)\mathcal{D}_{0})$. Therefore, from 
\cite[Theorem 1.1]{HMX}, for small enough $\beta\ll 1$, 
it should be at least $1$. In other words, 
$(\X_{0},\mathcal{D}_{0})$ is log canonical so that it is 
log K-semistable with respect to any polarization, by \cite[6.3]{OS}
($=$Theorem~\ref{OSrev}). Therefore, from the affine-linearity of the log 
Donaldson-Futaki invariant with respect to the coefficient of the boundary, 
it follows that $((\X_{0},\mathcal{D}_{0}),\mathcal{L}_{0})$ is 
strongly open K-polystable. 

The uniqueness assertion in the case \eqref{assum1} 
follows from \cite{BX}, while it in the case \eqref{assum2} follows from 
the Gromov-Hausdorff approach \cite{DS,CDS,Tia}. 
\end{proof}

\begin{Rem}
For general punctured family of polarized log Calabi-Yau pairs 
$((\X^{*}, \mathcal{D}^{*}),\mathcal{L}^{*})\to \Delta^{*}=\Delta\setminus 
\{0\}$, 
the uniqueness of strongly open K-polystable reduction as in the above 
Theorem~\ref{ststr} case 
does not necessarily hold. 
\end{Rem}
We discuss related developments in 
\cite{mod}. 


\bigskip

\footnotesize 
\noindent
Email address: \footnotesize 
{\tt yodaka@math.kyoto-u.ac.jp} \\
Affiliation: Department of Mathematics, Kyoto university, Japan  \\

\end{document}